\documentclass[a4paper,11pt]{amsart}
\usepackage[active]{srcltx}
\usepackage{latexsym,amssymb}
\usepackage{stmaryrd}
\usepackage{enumerate}
\usepackage{stackrel}
\usepackage[T1]{fontenc}
\usepackage[latin1]{inputenc}
\usepackage{xcolor}
\usepackage{hyperref}
\usepackage{pgf,tikz}
\usetikzlibrary{bending,positioning,arrows,automata,snakes,shapes}
\usepackage[all,pdf]{xy}
\usepackage{subfig}
\usepackage{mathtools}

\renewcommand{\ge}{\geqslant}
\renewcommand{\le}{\leqslant}
\let\op=\llbracket
\let\cl=\rrbracket
\def\pv#1{\ensuremath{\mathsf{#1}}}

\def\Om#1#2{\ensuremath{\overline{\Omega}_{#1}{\pv{#2}}}}

\newcommand\image{\mathop{\mathrm{Im}}}
\let\cal=\mathcal
\def\Cl#1{\ensuremath{\cal{#1}}}
\newcommand\malcev{%
\mathbin{\hbox{$\bigcirc$\kern-9pt\raise1.2pt\hbox{\scriptsize$m$}\,}}}
\newcommand\smalcev{%
\mathbin{\hbox{$\bigcirc$\kern-8pt\raise1.2pt\hbox{\tiny$m$}\,}}}
\newcommand{\relm}{\mathrel{\mathrlap{\hspace{2mm}\circ}{\longrightarrow}}}

\newtheorem{Thm}{Theorem}[section]
\newtheorem{Prop}[Thm]{Proposition}
\newtheorem{Lemma}[Thm]{Lemma}
\newtheorem{Cor}[Thm]{Corollary}

\theoremstyle{definition}

\newtheorem{eg}[Thm]{Example}

\begin{document}

\title[Locally countable pseudovarieties]{Locally countable
  pseudovarieties}

\author{J. Almeida}%
\address{CMUP, Dep.\ Matem\'atica, Faculdade de Ci\^encias,
  Universidade do Porto, Rua do Campo Alegre 687, 4169-007 Porto,
  Portugal} \email{jalmeida@fc.up.pt}

\author{O. Kl\'\i ma}%
\address{Dept.\ of Mathematics and Statistics, Masaryk University,
  Kotl\'a\v rsk\'a 2, 611 37 Brno, Czech Republic}%
\email{klima@math.muni.cz}

\begin{abstract}
  The purpose of this paper is to contribute to the theory of
  profinite semigroups by considering the special class consisting of
  those all of whose finitely generated closed subsemigroups are
  countable, which are said to be locally countable. We also call
  locally countable a pseudovariety \pv V (of finite semigroups) for
  which all pro-\pv V semigroups are locally countable. We investigate
  operations preserving local countability of pseudovarieties and show
  that, in contrast with local finiteness, several natural operations
  do not preserve it.
  We also investigate the relationship of a finitely generated
  profinite semigroup being countable with every element being
  expressable in terms of the generators using multiplication and the
  idempotent (omega) power. The two properties turn out to be
  equivalent if there are only countably many group elements, gathered
  in finitely many regular \Cl J-classes. We also show that the
  pseudovariety generated by all finite ordered monoids satisfying the
  inequality $1\le x^n$ is locally countable if and only if $n=1$.
\end{abstract}


\makeatletter%
\@namedef{subjclassname@2010}{%
  \textup{2010} Mathematics Subject Classification}%
\makeatother


\maketitle

\tableofcontents

\section{Introduction}
\label{sec:intro}

With the advent of modern computers, several mathematical models of
computation emerged. One of the simplest and most successful in
applications is that of a finite automaton, which computes regular
languages. The transition semigroup of the minimal automaton of a
regular language, which is also its syntactic semigroup, provides an
alternative purely algebraic computation model. This connection with
algebra turned out to be rather fruitful and eventually led to a
general framework for the translation of combinatorial problems on
classes of regular languages to algebraic problems about finite
semigroups, which is provided by Eilenberg's correspondence
\cite{Eilenberg:1976}. On the algebraic side, one considers so-called
pseudovarieties of semigroups, classes of finite semigroups closed
under taking homomorphic images, subsemigroups, and finite direct
products.

Unlike varieties, pseudovarieties do not in general possess free
members. To capture them, one needs to consider more general
structures, namely profinite semigroups, which are inverse limits of
finite semigroups. Relatively free profinite semigroups are then
realized naturally as topological spaces namely as the Stone duals of
Boolean algebras of regular languages
\cite[Theorem~3.6.1]{Almeida:1994a} (see also \cite{Pippenger:1997}).
In fact, the algebraic structure of such profinite semigroups is also
captured by duality theory
\cite{Gehrke&Grigorieff&Pin:2008,Gehrke:2016a,Almeida&Klima:2019d}.
Particularly in view of Reiterman's theorem \cite{Reiterman:1982},
that provides a description of pseudovarieties in terms of formal
equalities of members of a profinite semigroup (so-called
pseudoidentities), profnite semigroups thus naturally came to play an
important role in the theory of pseudovarieties of semigroups, which
remains the core of the theory of finite semigroups
\cite{Rhodes&Steinberg:2009qt}.

A very special type of pseudovarieties is that of locally finite
pseudovarieties. They may be characterized by the existence of a
finite bound on the cardinality of members on a given number of
generators, or by the finiteness of the finitely generated relatively
free profinite semigroups. Locally finite pseudovarieties constitute a
particularly well behaved special case, for instance when dealing with
the important operation of semidirect product (see
\cite[Section~3.7]{Rhodes&Steinberg:2009qt}).

In this paper, we consider a generalization of locally finite
pseudovarieties: locally countable pseudovarieties. These are
pseudovarieties whose finitely generated relatively free profinite
semigroups are countable. They turn out to have an important property,
namely that closed congruences on their finitely generated relatively
free profinite semigroups are profinite \cite{Almeida&Klima:2019d}, a
property that holds for arbitrary profinite groups but that fails in
general for profinite semigroups. The authors came across this
property in connection with the conjectured completeness of a proof
scheme for pseudoidentities \cite{Almeida&Klima:2017a}.

It should be noted that an infinite countably generated profinite
semigroup is either countable or has the cardinality of the continuum.
This follows from the Cantor-Bendixson Theorem (see
\cite[Theorem~6.4]{Kechris:1995}), which holds in every complete
metric space with a countable dense subset. In the case of a finitely
generated profinite semigroup $S$, one may say more precisely that, if
$S$~is uncountable, then it is the union of a countable set with a
closed set homeomorphic with the Cantor set. In particular, from the
point of view of cardinality, locally countable pseudovarieties should
deserve special attention.

There is a duality significance of a Stone space being countable that
somehow adds motivation to considering this cardinality condition.
Indeed, the Boolean algebras whose dual spaces are countable have been
characterized by Day~\cite[Section~4]{Day:1967} as the countable
superatomic Boolean algebras, where superatomic means that every
homomorphic image is atomic.

We proceed to describe briefly the organization and main contributions
of this paper. Section~\ref{sec:prelims} provides the preliminary
material on semigroups used in the rest of the paper. In
Section~\ref{sec:loc-count-pvs} we start the investigation of locally
countable pseudovarieties of semigroups; in particular, we discuss how
they behave with respect to the operations of join and semidirect
product. The Mal'cev product is considered in
Section~\ref{sec:Malcev}, where a profinite characterization is given
that plays a role in later sections. A well known and key result is
that the set of locally finite pseudovarieties is closed under Mal'cev
product, a fact that comes from Brown's finiteness theorem
\cite{Brown:1971}. One of the several proofs in the literature of
Brown's theorem is due to Simon \cite{Simon:1989} and uses his
Factorization Forest Theorem. This important theorem turns out to have
applications also in the study of locally countable pseudovarieties
that are based on a theorem presented in Section~\ref{sec:FFT}
concerning algebraic generation of a finitely generated profinite
semigroup. In Section~\ref{sec:LG}, we discuss profinite semigroups
with only finitely many regular \Cl J-classes. In particular, we show
that such a finitely generated profinite semigroup is countable if and
only if it has only countably many group elements, in which case it is
algebraically generated by the finite generating set together with the
idempotents. The special case where there are only finitely many
idempotents is considered in Section~\ref{sec:IE}, where it is shown
that the Mal'cev product of a locally countable pseudovariety of
semigroups with only one idempotent with a locally finite
pseudovariety is again locally countable; moreover, in such a case, we
show that elements of finitely generated relatively free profinite
semigroups can be obtained from the generators by using only product
and the $\omega$-power, without the need to nest the latter operation,
a statement that generalizes a very useful result of the first author
\cite{Almeida:1990b}. The Mal'cev product of a locally countable
pseudovariety with a locally finite pseudovariety in general is
considered in Section~\ref{sec:malcev-locfin}, where it is shown that
if the second factor consists of nilpotent semigroups, then the
resulting Mal'cev product is locally countable. By examining the atoms
in the lattice of pseudovarieties containing non-nilpotent semigroups,
we show in Section~\ref{sec:examples} that the result of
Section~\ref{sec:malcev-locfin} fails as soon as the second factor
contains a non-nilpotent semigroup. Finally, in
Section~\ref{sec:A-ESl} we use aperiodic inverse semigroups to show
that certain pseudovarieties of block groups are not locally
countable, thus answering some questions raised in our paper
\cite{Almeida&Klima:2017b}. In particular, we show that the
pseudovariety generated by all finite ordered monoids satisfying the
inequality $1\le x^n$ is not locally countable whenever $n\ge2$.

\section{Preliminaries}
\label{sec:prelims}

As the remainder of the paper is concerned with semigroups, we gather
in this section the required preliminary material on that subject. For
general background, the reader is referred to
\cite{Almeida:1994a,Rhodes&Steinberg:2009qt}.

A fundamental tool in semigroup theory is given by the so-called
Green's relations \cite{Green:1951}. They are binary relations
concerning the ideal structure of a semigroup $S$ and are defined as
follows. For $s,t\in S$, write $s\le_{\Cl R}t$ if $s$ belongs to the
principal right ideal generated by~$t$. Replacing right by left or
two-sided, one obtains respectively the relations $\le_{\Cl L}$ and
$\le_{\Cl J}$. The intersection of the relations $\le_{\Cl R}$ and
$\le_{\Cl L}$ is denoted $\le_{\Cl H}$. Each of the above relations
$\le_{\Cl K}$ is a quasi-order on~$S$ and the corresponding
equivalence relation ${\le_{\Cl K}}\cap{\ge_{\Cl K}}$ is denoted \Cl
K. The relations \Cl R and \Cl L commute under composition and so $\Cl
D=\Cl R\Cl L$ is also an equivalence relation on~$S$. In particular,
every \Cl D-class $D$ is a union of \Cl R-classes and also a union of
\Cl L-classes, where the intersection of each \Cl R-class with each
\Cl L-class within $D$ is an \Cl H-class. It turns out that the
relations \Cl D and \Cl J coincide in each compact semigroup (see, for
instance, \cite[Proposition~3.1.10]{Rhodes&Steinberg:2009qt}).

By a \emph{pseudovariety} we mean a class of finite semigroups that is
closed under taking homomorphic images, subalgebras and finite direct
products. For a pseudovariety \pv V, a \emph{pro-\pv V semigroup} is a
subdirect product of semigroups from~\pv V. In the case of the
pseudovariety \pv S of all finite semigroups, a pro-\pv S semigroup is
simply called a \emph{profinite semigroup}. A topological semigroup is
\emph{locally countable} if every finitely generated closed
subsemigroup is countable. Recall that a pseudovariety \pv V is
\emph{locally finite} if every finitely generated pro-\pv V semigroup
is finite. We say that a pseudovariety \pv V is \emph{locally
  countable} if every pro-\pv V semigroup is locally countable. Of
course, a locally finite pseudovariety is also locally countable. Such
notions may be extended to a general algebraic signature by simply
replacing semigroups by algebras in that signature. Since we only deal
with semigroups in the rest of the paper, there is no advantage in
formulating such definitions in a more general context.

The classes of all finite semigroups in which the relations \Cl J and
\Cl R are trivial (meaning that they are reduced to the equality
relation) are denoted, respectively, \pv J and \pv R. On the other
hand, the class of all \Cl H-trivial finite semigroups is denoted \pv
A. All three are pseudovarieties, the reason for the notation \pv A
being that the members of~\pv A are the finite \emph{aperiodic}
semigroups, that is, finite semigroups in which all subgroups are
trivial.

An element $s$ of a semigroup $S$ is \emph{regular} if there exists
$s'\in S$ such that $s=ss's$. It follows immediately from the
definitions that an element is regular if and only if its \Cl R-class
(respectively its \Cl L-class) contains an idempotent. Hence, a \Cl
D-class consists of regular elements if and only if it contains a
regular element, if and only if it contains an idempotent in each \Cl
R-class and in each \Cl L-class. Such \Cl D-classes are said to be
\emph{regular}. The \Cl H-classes of~$S$ that contain idempotents are
exactly the maximal subgroups of~$S$.

Profinite semigroups with only one \Cl J-class are said to be
\emph{completely simple}. A profinite semigroup with zero whose
nonzero elements form a regular \Cl J-class is said to be
\emph{completely 0-simple}. The finite completely simple semigroups
constitute a pseudovariety, denoted \pv{CS}. Those whose subgroups lie
in a pseudovariety of groups \pv H form a subpseudovariety of~\pv{CS},
denoted \pv{CS(H)}.

Given an ideal $I$ of a semigroup $S$, the union of the equality
relation on~$S$ with the universal relation on~$I$ is a congruence on
$S$. The corresponding quotient semigroup is denoted $S/I$ and is
called the \emph{Rees quotient} of $S$ by~$I$.

For a semigroup $S$, denote by $E(S)$ the set of its idempotent
elements.

Several operators may be defined on pseudovarieties. Let us consider
first some unary operators, defined on a pseudovariety \pv V. We
denote \pv{DV} the class of all finite semigroups whose regular \Cl
D-classes constitute subsemigroups belonging to~\pv V. The class
\pv{EV} consists of all finite semigroups $S$ such that the
subsemigroup generated by $E(S)$ belongs to~\pv V. For \pv{LV}, we
take the class of all finite semigroups $S$ such that, for every $e\in
E(S)$, the subsemigroup $eSe=\{ese:s\in S\}$ belongs to~\pv V. The
definition of the block operator \pv B is a bit more complicated.
Given a regular \Cl D-class $D$ of a finite semigroup~$S$, we consider
the smallest equivalence relation $\beta$ on the idempotents of~$D$
for which \Cl L or \Cl R-equivalent idempotents are equivalent. For
each $\beta$-class~$C$, we consider the subsemigroup $T$ generated by
the union of the \Cl H-classes containing members of $C$ and the ideal
$I=\{t\in T: t<_{\Cl J}e\}$ of~$T$, where $e\in C$ is arbitrary. The
Rees quotient $T/I$ is called a \emph{block} of~$D$. We denote \pv{BV}
the class of all finite semigroups whose blocks belong to~\pv V. It is
well known that \pv{BV}, \pv{DV}, \pv{EV} and \pv{LV} are
pseudovarieties.

Three binary operators on pseudovarieties play a special role in
finite semigroup theory. They correspond to classical algebraic
constructions. Given a semigroup $S$, we denote by $S^1$ the monoid
obtained from $S$ by adding an identity element unless $S$ is already
a monoid; we also denote by $\mathrm{End}(S)$ the monoid of all
endomorphisms of~$S$. Given another semigroup $T$ and a monoid
homomorphism $\varphi:S^1\to\mathrm{End}(T)$, we denote $\varphi(t)(s)$
by $\vphantom{|}^ts$. The Cartesian product $S\times T$ is then a
semigroup under the operation $(s,t)(s',t')=(s\,\vphantom{|}^ts',tt')$
which is denoted $S*T$ and is called the \emph{semidirect product} of
$S$ and $T$ determined by~$\varphi$.

Given a class \Cl C of finite semigroups, there is a smallest
pseudovariety containing it, which is called the pseudovariety
\emph{generated} by~\Cl C. Let \pv V and \pv W be two pseudovarieties.
The \emph{join} $\pv V\vee\pv W$ is the pseudovariety generated by
$\pv V\cup\pv W$ and may be thought as the pseudovariety generated by
all direct products $S\times T$ with $S\in\pv V$ and $T\in\pv W$. The
\emph{semidirect product} $\pv V*\pv W$ is the pseudovariety generated
by the class of all semidirect products $S*T$ with $S\in\pv V$ and
$T\in\pv W$. The \emph{Mal'cev} product $\pv V\malcev\pv W$ is the
pseudovariety generated by the class of all finite semigroups $S$ such
that there exists a homomorphism $\varphi:S\to T$ with $T\in\pv W$ and
$\varphi^{-1}(e)\in\pv V$ for every $e\in E(S)$.

The following definitions may be considered for arbitrary signatures
but we recall them here only in the context of semigroups required for
the remainder of the paper. For the general setting, see
\cite{Almeida:2003cshort,Almeida&Costa:2015hb}. We say that a
continuous mapping $\varphi:X\to S$ into a topological semigroup $S$
is a \emph{generating mapping} if the subsemigroup of $S$ generated by
$\varphi(X)$ is dense; in this case, we also say that $S$ is
\emph{$X$-generated}.

Let \pv V be a pseudovariety and let $X$ be a topological space. We
say that a continuous function $\varphi:X\to S$ into a pro-\pv V
semigroup $S$ defines $S$ as a \emph{free pro-\pv V semigroup
  over~$X$} if it has the following universal property: for every
continuous mapping $\psi:X\to T$ into another pro-\pv V semigroup $T$,
there is a unique continuous homomorphism $\hat{\psi}:S\to T$ such
that $\hat{\psi}\circ\varphi=\psi$. Standard arguments show that such
a pro-\pv V semigroup $S$ is unique up to homeomorphic isomorphism
respecting the choice of generators and it is denoted \Om XV. In case
$X$ is a finite set, we generally consider $X$ as a discrete space
when referring to~\Om XV. Moreover, when we write \Om nV for a
positive integer $n$, we mean \Om XV where $X=\{1,\ldots,n\}$.

The existence of free pro-\pv V semigroups over an arbitrary
topological space $X$ may be established by considering inverse limits
of $X$-generated semigroups from~\pv V. In particular, \Om XV is
$X$-generated. The generating mapping $\varphi:X\to\Om XV$ is also
called the \emph{natural mapping}. If \pv V is not the trivial
pseudovariety \pv I, consisting only of singleton semigroups, and $X$
is a profinite space, then the natural mapping $\varphi$ into~\Om XV is
injective and we usually identify each $x\in X$ with $\varphi(x)$.
Elements of \Om XV are sometimes called \emph{pseudowords}. Elements
of the subsemigroup of~\Om XV generated by $X$ are said to be
\emph{finite}, the remaining elements being called \emph{infinite}. A
word $u=x_1\cdots x_n$ with each $x_i\in X$ is said to be a
\emph{subword of the pseudoword} $v\in\Om XV$ if there is a
factorization $v=v_0x_1v_1\cdots x_nv_n$ with each $v_i$ in $(\Om
XV)^1$.

Note that the pseudovariety \pv V is locally countable if and only if
\Om nV is countable for every positive integer~$n$.

For a compact semigroup $S$ and an element~$s\in S$, it is well known
that the closed subsemigroup generated by~$s$ contains a unique
idempotent, which we denote $s^\omega$. In case $S$ is a finite
semigroup, $s^\omega$ is the unique idempotent power of~$s$.

Given a finite set $X$ and a pseudovariety \pv V, each element
$w\in\Om XV$ may be viewed as an operation $w_S:S^X\to S$ on each
pro-\pv V semigroup $S$ as follows: given a function $\psi\in S^X$ and
the natural mapping $\varphi:X\to\Om XV$, we let $w_S=\hat{\psi}(w)$,
where $\hat{\psi}$ is the unique continuous homomorphism such that the
following diagram commutes:
\begin{equation*}
  \label{eq:universal}
  \xymatrix{
    X \ar[r]^\varphi \ar[rd]_\psi
    & \Om XV \ar[d]^{\hat{\psi}} \\
    & S
  }
\end{equation*}
It may be shown that in this way $S$ becomes a profinite
$\Omega$-algebra where $\Omega_0=\emptyset$ and $\Omega_n=\Om nV$ for
$n\ge1$. Another convenient signature is that reduced to the
$\omega$-power and multiplication, for which every profinite semigroup
has thus a natural structure. The elements of the subalgebra of~\Om XV
generated by $X$ in this signature are called $\omega$-words.

A formal equality $u=v$ of elements of~\Om XV is said to be a
\emph{\pv V-pseudoidentity} over~$X$. It is said to be
\emph{satisfied} in the pro-\pv V semigroup $S$ if the equality
$u_S=v_S$ holds. The class of all semigroups from \pv V satisfying a
set $\Sigma$ of \pv V-pseudoidentities is a subpseudovariety of~\pv V
denoted $\op\Sigma\cl_{\pv V}$ and it is said to be \emph{defined}
by~$\Sigma$. By Reiterman's Theorem \cite{Reiterman:1982}, every
subpseudovariety is defined by some set of \pv V-pseudoidentities. In
case \pv V is the pseudovariety \pv S of all finite semigroups, then
we omit reference to~\pv V. The elements of~$X$ are often called
\emph{variables}. We may sometimes write $u=1$ as an abbreviation for
the pair of pseudoidentities $ux=x=xu$, where $x$ is a new variable. A
similar convention applies to $u=0$, which abbreviates $ux=u=xu$.

Given a residually finite discrete semigroup $S$, we may define a
metric on $S$ by letting $d(s,s)=0$ and, for distinct $s,s'\in S$,
$d(s,s')=2^{-r(s,s')}$ where $r(s,s')$ is the minimum cardinality
of a finite semigroup $T$ for which there is a homomorphism
$\varphi:S\to T$ such that $\varphi(s)\ne\varphi(s')$. Note that the
multiplication of $S$ is uniformly continuous with respect to~$d$. The
completion of the metric space $(S,d)$ has therefore a natural
structure of topological semigroup. It is called the \emph{profinite
  completion} of~$S$ and it is denoted $\widehat{S}$. In the
particular case where $S=X^+$ is the free semigroup generated by a
discrete set $X$, $\widehat{X^+}$ may be shown to be a free profinite
semigroup over~$X$. In particular, every element of~\Om XS is the
limit of some sequence of~words.

Denote by $\widehat{\mathbb{N}}$ the profinite completion of the
additive monoid $\mathbb{N}$ of natural numbers. It may be thought of
as the monoid $(\Om1S)^1$, written additively. If we denote $x$ the
generator $1$ of \Om1S, the identification between
$\widehat{\mathbb{N}}$ sends $0$ to the identity element~$1$ and each
positive $n\in\mathbb{N}$ to $x^n$. In general, we write $x^\alpha$
for the element of $(\Om1S)^1$ corresponding
to~$\alpha\in\widehat{\mathbb{N}}$.

Many pseudovarieties play a role in the sequel. For the moment, for
the sake of example, we introduce the following:
\begin{itemize}
\item $\pv{Sl}=\op x^2=x, xy=yx\cl$ is the pseudovariety of all finite
  \emph{semilattices};
\item $\pv N=\op x^\omega=0\cl=\bigcup_{n\ge1}\op x_1\cdots x_n=0\cl$
  is the pseudovariety of all finite \emph{nilpotent} semigroups;
\item $\pv G=\op x^\omega=1\cl$ is the pseudovariety of all finite
  groups;
\item $\pv K=\op x^\omega y=x^\omega\cl$ is the pseudovariety of all
  finite semigroups in which idempotents are left zeros;
\item $\pv{IE}=\op x^\omega=y^\omega\cl$ is the pseudovariety of all
  finite semigroups with a unique idempotent.
\end{itemize}
Further pseudovarieties will be introduced as needed.

\section{Locally countable pseudovarieties}
\label{sec:loc-count-pvs}

We proceed with some examples of locally countable pseudovarieties of
semigroups.

It was conjectured by I. Simon and proved by the first author that the
pseudovariety \pv J is locally countable
\cite[Corollary~3.4]{Almeida:1990b}. In particular, we deduce that the
pseudovariety \pv J satisfies the strong form of the conjecture
of~\cite{Almeida&Klima:2017a}.

For groups, local countability does not provide a new class of
pseudovarieties.

\begin{Thm}
  \label{t:Zelmanov}
  The following are equivalent for a pseudovariety of groups~\pv H:
  \begin{enumerate}[(i)]
  \item\label{item:Z-0} \Om1H is countable;
  \item\label{item:Z-1} \pv H is locally countable;
  \item\label{item:Z-2} \pv H is locally finite;
  \item\label{item:Z-3} \pv H satisfies some identity of the form
    $x^n=1$, where $n$~is a positive integer;
  \item\label{item:Z-4} \Om1H is finite.
  \end{enumerate}
\end{Thm}

\begin{proof}
  $(\text{\ref{item:Z-3}})\Leftrightarrow(\text{\ref{item:Z-4}})$ and
  $(\text{\ref{item:Z-2}})\Rightarrow(\text{\ref{item:Z-1}}) %
  \Rightarrow(\text{\ref{item:Z-0}})$ are obvious.

  $(\text{\ref{item:Z-3}})\Rightarrow(\text{\ref{item:Z-2}})$ follows from
  Zelmanov's solution of the restricted Burnside problem
  \cite{Zelmanov:1991}.

  $(\text{\ref{item:Z-0}})\Rightarrow(\text{\ref{item:Z-3}})$ Suppose
  that \Om1H is countable and assume that (\ref{item:Z-3})~fails.
  Then, \pv H must contain groups of arbitrarily large exponent. If
  the number of primes appearing in such exponents is finite, then
  there is a prime $p$ such that the cyclic group of order $p^n$
  belongs to~\pv H for every $n\ge1$. But, the inverse limit of such
  cyclic groups is the additive group of the ring $\mathbb{Z}_p$ of
  $p$-adic integers, which is uncountable and a homomorphic image of
  \Om 1H, contradicting the hypothesis (\ref{item:Z-0}). Hence, there
  is an infinite set of primes $\{p_1,p_2,\ldots\}$ such that
  $\mathbb{Z}/p_i\mathbb{Z}$ belongs to~\pv H. For each
  $\alpha\in\widehat{\mathbb{N}}$, consider the pseudovariety
  $\pv{Ab}_\alpha=\op x^\alpha=1, xy=yx\cl$. Let $\alpha$ be an
  accumulation point of the sequence $(p_1\cdots p_n)_n$ in
  $\widehat{\mathbb{N}}$. By the Chinese Remainder Theorem, it is easy
  to show that the mapping $\Om 1{Ab}_\alpha\to\prod_{i=1}^\infty\Om
  1{Ab}_{p_i}$ induced by the natural projections $\Om
  1{Ab}_\alpha\to\Om 1{Ab}_{p_i}$ is a bijection. This again
  contradicts the assumption that \pv H satisfies (\ref{item:Z-0})
  since $\pv{Ab}_\alpha$ is contained in~\pv H.
\end{proof}

Actually, a more general result is true: every profinite countable
group is finite. Several proofs of this fact may be given. A direct
proof can be found in~\cite[Proposition~2.3.1]{Ribes&Zalesskii:2010}.
Other proofs are obtained by applying various results.
In~\cite[Section~2.3]{Vannacci:2015PhD}, one can find immeadiate
proofs based on the Baire Category Theorem or using the Haar measure.
Yet another proof is obtained by considering the set of isolated
points of a profinite countable group. By symmetry, if there is such a
point then all points are isolated so that, by compactness, the group
is finite. Otherwise, by the Cantor-Bendixson Theorem, the group is
uncountable. In the proof of Theorem~\ref{t:Zelmanov}, thanks to
Zelmanov's theorem, we only need to deal with the cyclic case.

\begin{eg}
  \label{sec:equational-loccount-not-locfin}
  Consider the equational pseudovariety $\pv N_2=\op x^2=0\cl$. Note
  that it is locally countable. For a finite alphabet $A$, the
  semigroup $\Om AN_2$ is the Rees quotient of $A^+$ by the ideal
  consisting of the words containing some square factor. Hence, $\Om
  AN_2$ is infinite whenever $|A|\ge3$ \cite{Thue:1912}. Thus, in
  contrast with the group case, a locally countable pseudovariety need
  not be locally finite, even if it is equational.
\end{eg}

The remainder of this section examines operations that preserve local
countability and, therefore, provide methods to produce many examples
of locally countable pseudovarieties.

\begin{Thm}
  \label{t:join}
  If \pv V and \pv W are locally countable pseudovarieties then so is
  $\pv V\vee\pv W$.
\end{Thm}

\begin{proof}
  Let $A$ be a finite set. Simply note that the natural projections
  \Om A{(V\vee W)} onto \Om AV and \Om AW yield an embedding of~\Om
  A{(V\vee W)} into the product $\Om AV\times\Om AW$, which is
  countable since the factors are countable.
\end{proof}

The next result is an immediate application of a representation
theorem for free profinite semigroups over a semidirect product of
pseudovarieties $\pv V*\pv W$ \cite{Almeida&Weil:1995a}. Indeed, as
has already been observed in~\cite{Almeida&Klima:2017a}, since $\Om
A{(V*W)}$ embeds in a certain semidirect product $\Om BV*\Om AW$,
where $B=(\Om AW)^1\times A$, we have the following theorem.

\begin{Thm}
  \label{t:star}
  If \pv V and \pv W are pseudovarieties of semigroups such that \pv V
  is locally countable and \pv W is locally finite, then $\pv V*\pv W$
  is also locally countable.
\end{Thm}

For example, the pseudovarieties $\pv J*\pv B$, where $\pv B=\op
x^2=x\cl$ is the pseudovariety of all finite bands, and $\pv J*\op
x^n=1\cl$ are locally countable. Consider the pseudovariety \pv R. It
is known that $\pv{Sl}*\pv J=\pv R$ \cite{Brzozowski&Fich:1984} and
that \pv R is not locally countable: in fact \pv R contains the
pseudovariety \pv K and even \Om2K is uncountable as its infinite
elements are in natural bijection with the infinite words over a
two-letter alphabet (cf.~\cite[Section~3.7]{Almeida:1994a}). Hence,
the assumptions on \pv V and \pv W may not be exchanged in the
hypothesis of Theorem~\ref{t:star}. Since it is well known and easy to
check that \pv R is closed under semidirect product, we also have the
equality $\pv J*\pv J=\pv R$ and so even the semidirect product of a
locally countable pseudovariety with itself may not be locally
countable.

For the Mal'cev product, the situation is considerably more
complicated and is analyzed in the next six sections.

\section{The Mal'cev product and profinite semigroups}
\label{sec:Malcev}

Recall from Section~\ref{sec:prelims} the definition of the Mal'cev
product of pseudovarieties in terms of generators. A comprehensive
description of the Mal'cev product involves the notion of relational
morphism.

By a \emph{relational morphism} $\mu:S\relm T$ between semigroups $S$
and $T$ we mean a subsemigroup $\mu$ of the direct product $S\times T$
such that the projection $\mu\to S$ on the first component is onto. In
other words, a relational morphism $S\relm T$ is a relation with
domain $S$ and values in~$T$ which is closed under multiplication.
Some authors prefer to emphasize the transformation character of a
relational morphism, which is viewed as a function associating subsets
of $T$ to elements of $S$, calling the set of pairs $(s,t)\in S\times
T$ such that $s$ is related with $t$ the \emph{graph} of the
relational morphism (cf.~\cite{Rhodes&Steinberg:2009qt}). We make no
such distinction as we view relations as sets of ordered pairs as is
common in set theory.

It is well known that the Mal'cev product \pv{V\malcev W} consists of
all finite semigroups $S$ for which there is a relational morphism
$\mu:S\relm T$ into some $T\in\pv W$ such that $\mu^{-1}(e)\in\pv V$
for every $e\in E(T)$. Our first result, which will be instrumental
later in the paper, provides a profinite version of this comprehensive
characterization of the Mal'cev product. We have not been able to find
it in such a simple form in the literature although results of a
similar nature can be found in~\cite{Pin&Weil:1996a} and also
in~\cite[Chapter~3]{Rhodes&Steinberg:2009qt} in a more general
setting.

\begin{Thm}
  \label{t:Malcev}
  Let $S$ be a profinite semigroup and let \pv V and \pv W be
  pseudovarieties. Then $S$ is pro-\pv{(V\malcev W)} if and only if
  there exists a closed relational morphism $\mu:S\relm T$ such that
  $T$ is a pro-\pv W semigroup and $\mu^{-1}(e)$ is a pro-\pv V
  semigroup for every $e\in E(T)$.
\end{Thm}

\begin{proof}
  Assuming that $S$ is a pro-\pv{(V\malcev W)} semigroup, we know that
  $S$ is an inverse limit $\varprojlim S_i$ of an inverse system of
  semigroups $S_i$ from \pv{V\malcev W}, where $i$ runs over a
  directed set $I$. For each $i\in I$, there exists a relational
  morphism $\mu_i:S_i\relm T_i$ with $T_i\in\pv W$ and
  $\mu_i^{-1}(e)\in\pv V$ for every $e\in E(T_i)$. Given a finite
  subset $F$ of~$I$, let
  $$U_F=\{(s,f)\in S\times\prod_{i\in F}T_i: %
  (\forall i\in F)\ (\varphi_i(s),f(i))\in\mu_i\},$$ %
  where $\varphi_i:S\to S_i$ is the natural continuous homomorphism.
  Note that $U_F$ is a closed subsemigroup of $S\times\prod_{i\in
    F}T_i$, hence it is a profinite semigroup. For $F_1\subseteq
  F_2\subseteq I$ with $F_2$ finite, dropping the components from
  $F_2\setminus F_1$, we obtain an onto continuous homomorphism
  $U_{F_2}\to U_{F_1}$. Consider the profinite semigroup
  $U=\varprojlim U_F$, where $F$ runs over the finite subsets of~$I$.
  On the other hand, dropping the $S$-component of the elements
  of~$U_F$, we obtain a continuous homomorphism $U_F\to T_F$ onto a
  subsemigroup of the product $\prod_{i\in F}T_i$. Similarly, we
  obtain homomorphisms $T_{F_2}\to T_{F_1}$ for $F_1\subseteq
  F_2\subseteq I$ with $F_2$ finite. Let $T=\varprojlim T_F$ where,
  again, $F$ runs over the finite subsets of~$I$. Since the $T_i$
  belong to~\pv W, so does each $T_F$, whence $T$~is a pro-\pv W
  semigroup. The natural projections $U\to S$ and $U\to T$ show that
  $U$ may be viewed as a closed relational morphism $U:S\relm T$. We
  claim that $U^{-1}(e)$ is a pro-\pv V semigroup for every idempotent
  $e\in E(T)$.

  Let $e\in E(T)$. Given $i\in I$, let $e_i$ be the natural projection
  of $e$ in $T_{\{i\}}\subseteq T_i$. By the choice of~$\mu_i$, the
  semigroup $\mu_i^{-1}(e_i)$ belongs to~\pv V. Let
  $V_i=\varphi_i^{-1}(\mu_i^{-1}(e_i))=U_{\{i\}}^{-1}(e_i)$. Note that
  $U^{-1}(e)=\bigcap_{i\in I}V_i$ is a (nonempty) closed subsemigroup
  of~$S$, whence a profinite semigroup. Given two distinct elements
  $s,s'\in U^{-1}(e)$, there exists $i\in I$ such that
  $\varphi_i(s)\ne\varphi_i(s')$. As $\varphi_i(s)$ and
  $\varphi_i(s')$ belong to the semigroup $\mu_i^{-1}(e_i)$ from~\pv
  V, this shows that $U^{-1}(e)$ is residually~\pv V. Hence,
  $U^{-1}(e)$ is a pro-\pv V semigroup. This completes the proof of
  the direct implication in the statement of the theorem.

  For the converse, suppose that there is a relational morphism
  $\mu:S\relm T$ as in the statement of the theorem. Let $\varphi:S\to
  R$ be a continuous homomorphism onto a finite semigroup $R$. Since
  $S$ is profinite, we need to show that $R$ belongs to~\pv{V\malcev
    W} (cf.~\cite[Lemma~3.2.2]{Rhodes&Steinberg:2009qt}). Note that
  $\varphi^{-1}\mu:R\relm T$ is a closed relational morphism. For each
  idempotent $e\in E(T)$,
  $(\varphi^{-1}\mu)^{-1}(e)=\varphi(\mu^{-1}(e))$ is a finite
  continuous homomorphic image of a pro-\pv V semigroup and,
  therefore, it belongs to~\pv V
  \cite[Proposition~3.7]{Almeida:2003cshort}. This shows that $R$
  satisfies the same hypothesis as~$S$. Thus, we may assume from
  hereon that $S$ is finite.

  Given $s\in S$, the set $F_s=\{s\}\times\mu(s)$ is the preimage
  of~$s$ in~$\mu$ under the first component projection $\pi_1:\mu\to
  S$; since $\pi_1$ is continuous, $F_s$ is a clopen subset of~$\mu$.
  The image of $F_s$ under the second component projection
  $\pi_2:\mu\relm T$ is precisely the set $\mu(s)$; since $\pi_2$ is
  both a closed (being a continuous mapping between compact spaces)
  and an open mapping (being a component projection), we deduce that
  $\mu(s)$ is a clopen subset of~$T$. Now, consider the subalgebra $B$
  generated by the sets $\mu(s)$ ($s\in S$) of the Boolean algebra of
  all clopen subsets of~$T$. Since $S$ is assumed to be finite, $B$~is
  finite and its atoms form a partition of~$T$ into clopen subsets. By
  Hunter's Lemma \cite{Hunter:1988}, such a partition admits a
  refinement by a clopen congruence. In other words, there is a
  continuous homomorphism $\varphi:T\to U$ onto a finite semigroup $U$
  such that, if $t_1,t_2\in T$ are such that
  $\varphi(t_1)=\varphi(t_2)$ then, for every $s\in S$, $t_1\in\mu(s)$
  if and only if $t_2\in\mu(s)$.

  Consider the relational morphism $\nu=\mu\varphi:S\relm U$ obtained
  by relation composition. For $e\in E(U)$, the set $\varphi^{-1}(e)$
  is a closed subsemigroup of~$T$, whence it has some idempotent $f$.
  Note that $s\in\nu^{-1}(e)$ if and only if
  $\varphi^{-1}(e)\cap\mu(s)\ne\emptyset$. By the choice of~$\varphi$,
  the latter condition is in fact equivalent to
  $\varphi^{-1}(e)\subseteq\mu(s)$ and so also to $f\in\mu(s)$. Hence,
  $\nu^{-1}(e)=\mu^{-1}(f)$ is a semigroup from~\pv V. Since $U\in\pv
  W$, this shows that $S$ belongs to~\pv{V\malcev W}, as required.
\end{proof}

In the special case of the profinite semigroup \Om A{(V\malcev W)}, we
may say a bit more. Since it is a pro-\pv{(V\malcev W)} semigroup, by
Theorem~\ref{t:Malcev} there is a closed relational morphism $\mu:\Om
A{(V\malcev W)}\relm T$ into a pro-\pv W semigroup $T$ such that
$\mu^{-1}(e)$ is pro-\pv V for every $e\in E(T)$. We may choose for
each $a\in A$ an element $t_a\in\mu(a)$. The closed subsemigroup
of~$\mu$ generated by the set $\{(a,t_a):a\in A\}$ is still a closed
relational morphism sharing the above property with~$\mu$, and so we
may assume that it coincides with~$\mu$. Moreover, for the second
component projection $\pi_2:\mu\to T$, the semigroup
$\pi_2^{-1}(e)=\mu^{-1}(e)\times\{e\}$ is pro-\pv V. By
Theorem~\ref{t:Malcev}, $\mu$ is a pro-\pv{(V\malcev W)} semigroup.
Since it is $A$-generated via the mapping $a\mapsto(a,t_a)$, it
follows that the first component projection $\pi_1:\mu\to\Om
A{(V\malcev W)}$ is an isomorphism. The composite
$\pi_1^{-1}\pi_2=\mu$ is, therefore, a continuous homomorphism
$\varphi:\Om A{(V\malcev W)}\to T$. Since $T$ is pro-\pv W, $\varphi$
factors through the natural projection $p:\Om A{(V\malcev W)}\to\Om
AW$, which is identical on generators, as $\varphi=\psi\circ p$:
$$\xymatrix{
  &\mu \ar[ld]_{\pi_1} \ar[rd]^{\pi_2}
  & \\
  \Om A{(V\malcev W)} \ar[rr]^\varphi \ar[rd]^p
  && T \\
  &\Om AW \ar[ru]^\psi }$$ %
Given an idempotent $e\in\Om AW$, by the commutativity of the above
diagram, we have the inclusion $p^{-1}(e)\subseteq\varphi^{-1}(\psi(e))$,
which entails that $p^{-1}(e)$ is a pro-\pv V semigroup. This
discussion proves the following consequence of Theorem~\ref{t:Malcev}.

\begin{Cor}
  \label{c:Malcev}
  The natural projection $p:\Om A{(V\malcev W)}\to\Om AW$ is such that
  $p^{-1}(e)$ is a pro-\pv V semigroup for every $e\in E(\Om AW)$.\qed
\end{Cor}

From Corollary~\ref{c:Malcev} and Theorem~\ref{t:Malcev}, it is
immediate to deduce the so-called Pin-Weil Basis Theorem for the
Mal'cev product \cite{Pin&Weil:1996a} (see also
\cite[Theorem~3.7.13]{Rhodes&Steinberg:2009qt}): if
$\{u_i(x_1,\ldots,x_{n_i})=v_i(x_1,\ldots,x_{n_i}): i\in I\}$ is a
basis of pseudoidentities for~\pv V, then
$$\{u_i(w_1,\ldots,w_{n_i})=v_i(w_1,\ldots,w_{n_i}): %
i\in I, %
\pv W\models w_1^2=w_1=\cdots=w_{n_i}\}$$ %
is a basis of pseudoidentities for~\pv{V\malcev W}.

\section{An application of the Factorization Forest Theorem}
\label{sec:FFT}

Already the preservation of local finiteness by the Mal'cev product of
semigroup pseudovarieties is a nontrivial result depending on a
finiteness theorem of Brown \cite{Brown:1971} which was first proved
using combinatorial methods. Several proofs of Brown's theorem are
available in the literature. See \cite[Theorem~4.2.4 and Notes to
Chapter~4]{Rhodes&Steinberg:2009qt} for an algebraic proof and several
references.

A particularly elegant proof of Brown's theorem is due to Simon
\cite{Simon:1989} using his Factorization Forest Theorem, which we
proceed to recall.

For a set $X$, let $\Cl F(X)$ be the set of all finite sequences
$(x_1,\ldots,x_n)$ of elements of $X$. The \emph{length} of the
sequence $s=(x_1,\ldots,x_n)$ is~$n$ and is denoted $|s|$. By a
\emph{factorization forest over $A$} we mean a pair $F=(X,d)$, where
$X$ if a subset of~$A^+$ and $d:X\to\Cl F(X)$ such that, for every
$x\in X$, $d(x)=(x_1,\ldots,x_n)$ implies $x=x_1\cdots x_n$.

Let $F=(X,d)$ be a factorization forest over~$A$. If $x\in X$ then we
say that the \emph{degree} of $x\in X$ is 0 if $|d(x)|=1$, while it is
$|d(x)|$ otherwise. The \emph{external elements} of~$F$ are the
elements of~$X$ of degree~0. The \emph{height} $h(x)$ of~$x$ is
defined recursively as follows: $h(x)=0$ if $x$ is external and
$h(x)=1+\max\{h(x_i): 1\le i\le n\}$ if $d(x)=(x_1,\ldots,x_n)$ with
$n>1$. The \emph{height} of $F$ is $h(F)=\sup\{h(x):x\in X\}$.

Let $f:A^+\to S$ be a semigroup homomorphism. A factorization forest
$F$ is \emph{Ramseyan modulo $f$} if, whenever the degree of $x$ is at
least 3 and $d(x)=(x_1,\ldots,x_n)$, we have
$f(x)=f(x_1)=\cdots=f(x_n)$ and $f(x)$ is idempotent. We say that $f$
\emph{admits} the factorization forest $F=(X,d)$ if $X=A^+$ and the
external set is~$A$.

\begin{Thm}[Factorization Forest Theorem~\cite{Simon:1990}]
  \label{t:FFT}
  Every homomorphism $f:A^+\to S$ into a finite semigroup $S$ admits a
  Ramseyan factorization forest of height at most $9|S|$.
\end{Thm}

Relaxing the linear bound in the above theorem to an exponential
bound, Simon~\cite{Simon:1992} has also produced a simplified proof
based on the Krohn-Rhodes Decomposition Theorem.

Our proof of the following result relies on the methods and ideas of
Simon's proof of Brown's theorem \cite[Theorem~5]{Simon:1989}.

\begin{Thm}
  \label{t:Brown-profinite}
  Let $\varphi:S\to T$ be a continuous homomorphism where $S$ is a
  profinite semigroup generated by a finite set $A$ and $T$ is a
  finite semigroup. Then the semigroup $S$ is algebraically generated
  by $A$ together with the elements that map to idempotents in~$T$.
\end{Thm}

\begin{proof}
  Let $\psi:\Om AS\to S$ be the continuous homomorphism mapping each
  generator to itself and let $f:A^+\to S$ be obtained by restricting
  $\varphi\circ\psi$ to~$A^+$. By Theorem~\ref{t:FFT}, $f$ admits a
  Ramseyan factorization forest $F=(A^+,d)$ of finite height $H$. For
  each integer $k\in[0,H]$, consider the subset $S_k$ of~$S$
  consisting of all elements of the form $\psi(w)$ where $w=\lim w_n$
  for a sequence of words $w_n\in A^+$ with $h(w_n)\le k$. Note that
  $S_H=S$. Let $E=E(T)$. By induction on~$k$, we show that every
  element of~$S_k$ is a (finite) product of elements of
  $A\cup\varphi^{-1}(E)$.

  Since the factorization forest $F$ is admitted by~$f$, the set of
  external elements of $F$ is $A$. As $A$ is a finite set, it follows
  that $S_0=A$. Assume now, inductively, that $k>0$ and $S_{k-1}$ is
  contained in the subsemigroup $\langle A\cup\varphi^{-1}(E)\rangle$
  generated by $A\cup\varphi^{-1}(E)$. Let $s\in S_k$ and let
  $(w_n)_n$ be a convergent sequence of words from $A^+$ of height at
  most $k$ such that $s=\psi(\lim w_n)$. We claim that $s\in\langle
  A\cup\varphi^{-1}(E)\rangle$.

  In case an infinite number of terms $w_n$ of the sequence have
  height less than~$k$, then $s\in S_{k-1}$ and the induction
  hypothesis yields the claim. On the other hand, if an infinite
  number of the $w_n$ have degree $2$, say $d(w_n)=(u_n,v_n)$, then
  $h(u_n)$ and $h(v_n)$ are both less than $k$. By compactness, there
  is a strictly increasing sequence of indices $n_r$ such that the
  sequences $(u_{n_r})_r$ and $(v_{n_r})_r$ converge. Let
  $s'=\psi(\lim u_{n_r})$ and $s''=\psi(\lim v_{n_r})$, so that
  $s=s's''$. Since $s'$ and $s''$ belong to $\langle
  A\cup\varphi^{-1}(E)\rangle$ by the induction hypothesis, it follows
  that so does $s$.

  Hence, we may assume that every $w_n$ has degree at least three.
  Since $F$ is Ramseyan modulo $f$, we conclude that $f(w_n)$ is an
  idempotent. Since $T$~is finite and $f$ is continuous, it follows
  that $\varphi(s)$ is also an idempotent, that is,
  $s\in\varphi^{-1}(E)$. This concludes the induction step and the
  proof.
\end{proof}

\section{Profinite semigroups with only finitely many regular
  \texorpdfstring{\Cl J}{J}-classes}
\label{sec:LG}

In this section, we investigate profinite semigroups with only
finitely many regular \Cl J-classes and their relationship with
pseudovarieties consisting of semigroups with only one regular \Cl
J-class, that is, subpseudovarieties of~\pv{LG}. We start with a
couple of general results about Green's relations in profinite
semigroups.

\begin{Prop}
  \label{p:Green}
  Let $S$ be a profinite semigroup and let \Cl K be any of Green's
  relations \Cl L, \Cl R, \Cl J, \Cl H or the corresponding
  quasiorders $\le_{\Cl L}$, $\le_{\Cl R}$, $\le_{\Cl J}$, and
  $\le_{\Cl H}$. Then two elements of~$S$ are \Cl K-related if and
  only if their images under every continuous homomorphism onto a
  finite semigroup are \Cl K-related.
\end{Prop}

\begin{proof}
  The proof is a more or less standard compactness argument which
  works more generally for a system of equations with given parameters
  (see \cite[Section~5.6]{Almeida:1994a}, where the argument is
  formulated only for relatively free profinite semigroups). For the
  sake of completeness, we present the proof in the case of~\Cl L. The
  other cases may be proved similarly. Since homomorphisms preserve
  the relation~\Cl L, we may consider a pair of elements $u,v\in S$
  such that $\varphi(u)$ and $\varphi(v)$ are \Cl L-related for every
  continuous homomorphism $\varphi:S\to T$ onto a finite
  semigroup~$T$. We may consider an inverse system $S_i$ of finite
  semigroups $S_i$ indexed by a directed set $I$ and onto
  homomorphisms $S_i\to S_j$ ($i\ge j$) whose inverse limit is
  isomorphic with~$S$. For each $i\in I$, let $\varphi_i:S\to S_i$ be
  the projection homomorphism. Let $X_i$ be the set of all pairs
  $(x,y)$ of elements of $S^1$ such that $\varphi_i(xu)=\varphi_i(v)$
  and $\varphi_i(yv)=\varphi_i(u)$. Since $\varphi_i$ is continuous,
  $X_i$ is a closed subset of~$S^1\times S^1$. On the other hand,
  since $I$ is a directed set and $X_i\subseteq X_j$ whenever $i\ge
  j$, every finite subfamily of the family of closed sets $(X_i)_{i\in
    I}$ has nonempty intersection. By compactness, the set
  $X=\bigcap_{i\in I}X_i$ is nonempty. Given a pair $(x,y)\in X$, we
  must have $xu=v$ and $yv=u$ since the two sides of each of these
  equations have the same image under every homomorphism $\varphi_i$.
  This shows that $u\mathrel{\Cl L}v$.
\end{proof}

The following is a simple application of Proposition~\ref{p:Green}.

\begin{Cor}
  \label{c:Green}
  Let $S$ be a profinite semigroup. Then each of Green's relations \Cl
  L, \Cl R, \Cl J, and \Cl H is a closed equivalence relation and the
  corresponding quotient topological space is profinite.
\end{Cor}

\begin{proof}
  Since the argument is similar for all the considered relations, we
  deal only with \Cl L. By Proposition~\ref{p:Green}, the relation
  $\Cl L_S$ on~$S$ is the intersection of the closed relations
  $(\varphi\times\varphi)^{-1}\Cl L_T$ where $\varphi$ runs over all
  continuous homomorphisms $S\to T$ onto finite semigroups $T$. Hence,
  $\Cl L_S$ is closed. Moreover, if $u,v\in S$ are not \Cl L-related
  then, again by Proposition~\ref{p:Green}, there exists a continuous
  homomorphism $\varphi:S\to T$ onto a finite semigroup $T$ such that
  $\varphi(u)$ and $\varphi(v)$ are not \Cl L-related. Since the \Cl L
  relation is preserved by homomorphisms, we obtain the following
  commutative diagram of continuous functions, where the vertical
  arrows are the natural mappings:
  $$\xymatrix{
    S \ar[r]^\varphi \ar[d] %
    & T \ar[d] \\
    S/\Cl L_S \ar@{-->}[r] %
    & T/\Cl L_T. }$$ %
  Thus, distinct points in the quotient space $S/\Cl L_S$ may be
  distinguished by continuous mappings onto finite discrete sets,
  which shows that $S/\Cl L_S$ is a profinite space.
\end{proof}

\begin{Prop}
  \label{p:1oc-count-in-LG}
  A countably generated pro-\pv{LG} semigroup is countable if and only
  if it has only countably many idempotents and its subgroups are
  finite.
\end{Prop}

\begin{proof}
  Let $S$ be a countably generated pro-\pv{LG} semigroup. If $S$
  countable then it certainly has only countably many idempotents and
  all its subgroups are countable, whence finite
  (cf.~Section~\ref{sec:loc-count-pvs}). For the converse, the
  hypothesis guarantees that the minimum ideal $K$ of~$S$ is countable
  since it is the union of the subgroups contained in it. On the other
  hand, since $S$ is pro-\pv{LG}, the Rees quotient $S/K$ is a
  countably generated pro-\pv N semigroup, whence countable. Thus, $S$
  is countable.
\end{proof}

The following result shows how \pv{LG} may be used to decompose
semigroups with only finitely many regular \Cl J-classes. It depends
on the well known fact that a profinite semigroup is pro-\pv{LG} if
and only if the only semilattice that embeds in it is the trivial one.

\begin{Prop}
  \label{p:finmany-reg-J-classes}
  Let $S$ be a profinite semigroup with only finitely many regular \Cl
  J-classes. Then, there is a continuous homomorphism $\varphi:S\to T$
  onto a finite semigroup such that $\varphi^{-1}(e)$ is a pro-\pv{LG}
  semigroup for every idempotent $e\in T$.
\end{Prop}

\begin{proof}
  Choose an element $s_J$ from each regular \Cl J-class $J$ of~$S$. By
  Proposition~\ref{p:Green}, there exists a continuous homomorphism
  $\varphi:S\to T$ onto a finite semigroup $T$ such that the
  restriction of the \Cl J-order of $S$ to the $s_J$ is isomorphic
  under $\varphi$ to the restriction of the \Cl J-order of $T$ to the
  $\varphi(s_J)$.

  We claim that $\varphi^{-1}(e)$ is pro-\pv{LG} for every idempotent
  $e\in T$. Indeed, from the choice of $\varphi$ it follows that there
  cannot be a pair of idempotents $f$ and $g$ in~$\varphi^{-1}(e)$
  such that $f<_{\Cl J}g$.
  This implies that $\varphi^{-1}(e)$ is a pro-\pv{LG} semigroup.
\end{proof}

We may now give a characterization of finitely generated profinite
semigroups with only finitely many regular \Cl J-classes in terms of
Mal'cev products.

\begin{Cor}
  \label{c:LG-malcev}
  Let $S$ be a finitely generated profinite semigroup. Then $S$ has
  only finitely many regular \Cl J-classes if and only if $S$ is a
  pro-\pv{(LG\malcev V)} semigroup for some locally finite
  pseudovariety \pv V.
\end{Cor}

\begin{proof}
  Suppose first that $S$ has only finitely many regular \Cl J-classes.
  Consider a continuous homomorphism $\varphi:S\to T$ given by
  Proposition~\ref{p:finmany-reg-J-classes} and let \pv V be the
  pseudovariety generated by the finite semigroup $T$. Then, by
  Theorem~\ref{t:Malcev}, $S$ is a pro-\pv{(LG\malcev V)} semigroup.

  Conversely, if $S$ is a pro-\pv{(LG\malcev V)} semigroup with \pv V
  locally finite, then Theorem~\ref{t:Malcev} provides a closed
  relational morphism $\mu:S\relm T$ into a pro-\pv V semigroup $T$.
  Since $S$ is finitely generated, we may assume that so is~$T$. Since
  \pv V is locally finite, it follows that $T$ is finite. Given a
  regular element $s\in S$, there exists $s'\in S$ such that $ss's=s$
  and $s'ss'=s'$. Take an idempotent $e$ in the subsemigroup
  $\mu(ss')$ of~$T$. Then, the idempotent $ss'$ belongs to
  $\mu^{-1}(e)$. Since there is only one regular \Cl J-class $J$
  in~$\mu^{-1}(e)$, as $\mu^{-1}(e)$ is pro-\pv{LG}, we conclude that
  $s$ belongs to the \Cl J-class of~$S$ containing~$J$. Hence, there
  are only finitely many regular \Cl J-classes in~$S$.
\end{proof}

Pseudovarieties of the form \pv{LG\malcev V} play a special role in
the semilocal theory of Rhodes (cf.\
\cite[Theorem~4.6.50]{Rhodes&Steinberg:2009qt}).

The following theorem gives factorizations for elements of a finitely
generated profinite semigroup with only finitely many regular \Cl
J-classes.

\begin{Thm}
  \label{t:finmany-reg-J-classes}
  If $S$ is a profinite semigroup generated by a finite set $A$ and
  $S$ has only finitely many regular \Cl J-classes then $S$ is
  algebraically generated by $A$ together with the group elements
  of~$S$. Moreover, in case $S$ has only countably many group
  elements, $S$ is algebraically generated by $A\cup E(S)$.
\end{Thm}

\begin{proof}
  By Proposition~\ref{p:finmany-reg-J-classes}, there is a continuous
  homomorphism $\varphi:S\to T$ onto a finite semigroup such that
  $\varphi^{-1}(e)$ is a pro-\pv{LG} semigroup for every idempotent
  $e\in T$. By Theorem~\ref{t:Brown-profinite}, $S$ is algebraically
  generated by~$A\cup\varphi^{-1}(E(T))$. By induction on the depth of
  an idempotent $e\in E(T)$ in the poset $E(T)/\Cl J|_{E(T)}$, we show
  that every element in $\varphi^{-1}(e)$ is a product of elements of
  $A$ and group elements. Indeed, given an idempotent $e\in T$, since
  $\varphi^{-1}(e)$~is pro-\pv{LG}, its elements are either group
  elements or products of elements of~$A\cup\varphi^{-1}(E_e)$, where
  $E_e$ denotes the set of all idempotents of~$T$ strictly \Cl
  J-above~$e$. By the induction hypothesis, in the latter case such
  elements are themselves products of elements of~$A$ and group
  elements of~$S$.

  In the special case where $S$ has only countably many group
  elements, we may choose a maximal subgroup $G_e$ of
  $\varphi^{-1}(e)$ for each idempotent $e\in E(T)$. By assumption,
  $G_e$ is countable, whence finite. There is, therefore, a continuous
  homomorphism $\psi:S\to U$ onto a finite semigroup $U$ that
  separates the points of~$G_e$ for every $e\in E(T)$.

  Now, let $g\in\varphi^{-1}(e)$ be a group element. Let $\varepsilon$
  be the idempotent in~$G_e$ and, in the semigroup $\varphi^{-1}(e)$,
  let $f$ be the idempotent in the \Cl R-class of $g$ that is \Cl
  L-equivalent to~$\varepsilon$. Since the product $\varepsilon gf$
  belongs to $G_e$ and $S$ is topologically generated by~$A$, there
  exists a product $w$ of elements of~$A$ such that
  $\psi(w)=\psi(\varepsilon gf)$ and $\varphi(w)=e$. As $\psi(G_e)$ is
  a group and $\psi(\varepsilon)$ is its idempotent, we have
  $\psi(w)=\psi(\varepsilon w\varepsilon)$. Since $\varphi^{-1}(e)$ is
  a pro-\pv{LG} semigroup, $\varepsilon w\varepsilon$ is an element
  of $G_e$. As $\psi$ separates the elements of $G_e$, we deduce from
  the equalities $\psi(\varepsilon
  w\varepsilon)=\psi(w)=\psi(\varepsilon gf)$ that $\varepsilon
  w\varepsilon=\varepsilon gf$. Taking into account that $g^\omega$ is
  an idempotent \Cl R-equivalent to~$f$ and $g^{\omega+1}=g$, it
  follows that
  $$f\varepsilon w\varepsilon g^\omega %
  =f\varepsilon gfg^\omega %
  =fgfg^\omega=g^{\omega+1}=g.$$ %
  This shows that $g$ is a product of elements of~$A\cup E(S)$ and, in
  view of the first part of the proof, completes the proof of the
  theorem.
\end{proof}

For the pseudovariety \pv S of all finite semigroups, since the
regular \Cl J-classes of~\Om A{DS} are characterized by their content
(cf.~\cite[Theorem~8.1.7]{Almeida:1994a}),
Theorem~\ref{t:finmany-reg-J-classes} applies to every finitely
generated pro-\pv{DS} semigroup. Hence,
Theorem~\ref{t:finmany-reg-J-classes} generalizes a result of Azevedo
and the first author \cite{Almeida&Azevedo:1987} (see
also~\cite[Section~8.1]{Almeida:1994a}).

\begin{Cor}
  \label{c:finmany-reg-J-classes-countmany-groupels}
  If $S$ is a finitely generated profinite semigroup with only
  finitely many regular \Cl J-classes and countably many group
  elements, then $S$ is countable.
\end{Cor}

\begin{proof}
  By Theorem~\ref{t:finmany-reg-J-classes}, $S$ is algebraically
  generated by a countable set and, therefore, it is countable.
\end{proof}

\section{Profinite semigroups with only finitely many idempotents}
\label{sec:IE}

In this section, we consider the special case of the setting of
Section~\ref{sec:LG} when there are only finitely many idempotents.

Recall that the pseudovariety \pv{IE} consists of all finite
semigroups with only one idempotent.

\begin{Prop}
  \label{p:1oc-count-in-IE}
  A subpseudovariety \pv V of\/ \pv{IE} is locally countable if and only
  if it is contained in $\pv N\vee\pv H$ for some locally finite
  pseudovariety \pv H of groups.
\end{Prop}

\begin{proof}
  Suppose that \pv V is locally countable. Then, so is the
  intersection $\pv H=\pv V\cap\pv G$. By Theorem~\ref{t:Zelmanov},
  the pseudovariety \pv H is locally finite. As observed
  in~\cite[Section~9.1]{Almeida:1994a}, it follows that $\pv
  V\subseteq\pv N\vee\pv H$.

  The converse follows from Theorem~\ref{t:join}.
\end{proof}

The following is an elementary but useful observation that holds more
generally for compact semigroups.

\begin{Lemma}
  \label{l:unique-minimal-idempotent-central}
  Let $S$ be a profinite semigroup and suppose that its minimum ideal
  has a unique idempotent $e$. Then $e$ is central in~$S$.
\end{Lemma}

\begin{proof}
  Let $s\in S$ and let $K$ be the minimum ideal of~$S$. As
  $(se)^\omega$ and $(es)^\omega$ are both idempotents in~$K$, they
  are equal to~$e$. Hence, we have
  $$se=s(es)^\omega=(se)^\omega s=es.\popQED$$
\end{proof}

The relationship between the cardinality of the set of idempotents of
a profinite semigroup and the pseudovariety \pv{IE} is considered in
the following result.

\begin{Thm}
  \label{t:E-finite}
  Let $S$ be a profinite semigroup.
  \begin{enumerate}[(i)]
  \item\label{item:E-finite-1} The set $E(S)$ is a singleton if and
    only if $S$ is pro-\pv{IE}.
  \item\label{item:E-finite-2} If $E(S)$ is finite then $S$ is
    pro-$(\pv{IE}\malcev\pv V)$ for some locally finite pseudovariety
    \pv V.
  \item\label{item:E-finite-3} If \pv V is a locally finite
    pseudovariety and $S$ is a finitely generated
    pro-$(\pv{IE}\malcev\pv V)$ semigroup, then $E(S)$ is finite.
  \end{enumerate}
\end{Thm}

\begin{proof}
  (\ref{item:E-finite-1}) If $S$ has two distinct idempotents then
  they may be separated by a continuous homomorphism onto a finite
  semigroup $T$. Since $T$ has at least two idempotents, it follows
  that $T$ is not in~\pv{IE}. Hence, $S$ is not pro-\pv{IE}.
  Conversely, if $S$ is pro-\pv{IE}, then it embeds in a product of
  semigroups with only one idempotent and so it has only one
  idempotent.

  (\ref{item:E-finite-2}) Now, suppose that the profinite semigroup
  $S$ has only finitely many idempotents. Since $S$ is residually
  finite, there is an idempotent separating continuous homomorphism
  $\varphi:S\to T$ onto a finite semigroup $T$.
  By~(\ref{item:E-finite-1}), since $\varphi^{-1}(e)$ has only one
  idempotent for each $e\in E(T)$, the profinite semigroup
  $\varphi^{-1}(e)$ is pro-\pv{IE}. Hence, by Theorem~\ref{t:Malcev},
  $S$~is pro-$(\pv{IE}\malcev\pv V(T))$, where $\pv V(T)$ is the
  (locally finite) pseudovariety generated by~$T$.

  (\ref{item:E-finite-3}) Suppose that $S$ is a finitely generated
  pro-$(\pv{IE}\malcev\pv V)$ semigroup where \pv V is a locally
  finite pseudovariety. To show that $S$ has finitely many
  idempotents, it suffices to show that so does $\Om A{(IE\malcev V)}$
  for every finite set~$A$. Consider the natural continuous
  homomorphism $\varphi:\Om A{(IE\malcev V)}\to\Om AV$. Then, for each
  idempotent $e\in E(\Om AV)$, the profinite semigroup
  $\varphi^{-1}(e)$ is pro-\pv{IE} by Corollary~\ref{c:Malcev}, whence
  it has only one idempotent by~(\ref{item:E-finite-1}). Since the
  image of every idempotent of $\Om A{(IE\malcev V)}$ is an idempotent
  of~\Om AV, there are no further idempotents to consider other than
  those from the $\varphi^{-1}(e)$ with $e\in E(\Om AV)$. Hence, $\Om
  A{(IE\malcev V)}$ has the same number of idempotents as the finite
  semigroup~\Om AV.
\end{proof}

Note that the hypothesis that $S$ is finitely generated may not be
dropped from the statement (\ref{item:E-finite-3}) of
Theorem~\ref{t:E-finite}. Indeed, if $A$ is the one-point
compactification of an infinite discrete set then the semigroup \Om
A{Sl} consists of idempotents and it is uncountable for, as it is
observed in~\cite[Example, Section~1.2]{Almeida&Weil:1995a}, \Om A{Sl}
is isomorphic with the semilattice of closed subsets of~$A$ under
union.

\begin{Thm}
  \label{t:IE(H)-m-locfin}
  Let \pv V be a locally countable subpseudovariety of~\pv{IE} and let
  \pv W be a locally finite pseudovariety. Then \pv{V\malcev W} is
  locally countable.
\end{Thm}

\begin{proof}
  By Proposition~\ref{p:1oc-count-in-IE}, there is a locally finite
  pseudovariety of groups \pv H such that $\pv V\subseteq\pv N\vee\pv
  H$. Let $A$ be a finite set and consider the natural continuous
  homomorphism $\varphi:\Om A{(V\malcev W)}\to\Om AW$, which is
  idempotent separating since the preimage of each idempotent belongs
  to~\pv{IE}. By Theorem~\ref{t:finmany-reg-J-classes}, we know that
  every element of $\Om A{(V\malcev W)}$ is a product of elements of
  $A$ and group pseudowords. Hence, to show that $\Om A{(V\malcev W)}$
  is countable, it suffices to show that it has only countably many
  group elements. In fact, we show that it has only finitely many
  group elements, that is, that every subgroup is finite. Since $\Om
  A{(V\malcev W)}$ has only finitely many idempotents, namely as many
  as \Om AW, that goal is achieved by showing, equivalently, that
  every regular \Cl J-class $J$ of $\Om A{(V\malcev W)}$ is finite.

  Consider the closed subsemigroup $S$ of $\Om A{(V\malcev W)}$
  generated by $J$. The restriction of~$\varphi$ to~$S$ is still an
  idempotent separating continuous homomorphism. The fact that the
  preimage of each idempotent belongs to~\pv{IE} guarantees that
  $\varphi(J)\cap\varphi(S\setminus J)=\emptyset$: if $u\in J$ and
  $s\in S$ are such that $\varphi(u)=\varphi(s)$, then there exists
  $v\in J$ such that $uv$ is an idempotent from~$J$ and so
  $\varphi(uv)=\varphi((sv)^\omega)$; since
  $\varphi^{-1}(\varphi(uv))$ is pro-\pv{IE}, it follows that
  $uv=(sv)^\omega$, which entails that $s\in J$. Hence, the profinite
  completely 0-simple semigroup given by the Rees quotient
  $S_0=S/(S\setminus J)$ admits an idempotent separating continuous
  homomorphism $\psi$ onto a finite semigroup $T$. If $s$ is an
  element of~$J$ but not a group element, then we claim that
  $\varphi(s)$ cannot be idempotent. Otherwise, we have
  $\varphi(s)=\varphi(s^\omega)=\varphi(0)$; since $\varphi$ preserves
  \Cl J-equivalence, all of~$J$ must be mapped to~$\varphi(0)$, which
  contradicts the choice of~$\varphi$. Hence, the preimage of each
  idempotent is a pro-\pv H group and, therefore, it is locally
  finite. By Brown's theorem, we conclude that $S_0$ is locally
  finite. Observing that taking the Rees quotient by the ideal
  $S\setminus J$ to obtain $S_0$ does not affect~$J$, we conclude
  that, to prove that $J$ is finite, it suffices to show that $S_0$~is
  finitely generated.

  Returning to $\Om A{(V\malcev W)}$, given an idempotent $e\in J$,
  the profinite semigroup $U=\varphi^{-1}(\varphi(e))$ is open. Since
  the subsemigroup of~$\Om A{(V\malcev W)}$ generated by $A$ is dense,
  we deduce that its intersection $V$ with~$U$ is dense in~$U$. Given
  $v\in V$, we have $e=v^\omega=(eve)^\omega$ since $e$ is central
  in~$U$ by Lemma~\ref{l:unique-minimal-idempotent-central}. Hence,
  $eve$ lies in the maximal subgroup $H_e$ containing $e$ and $H_e$ is
  generated by the elements of the form $eve$ with $v\in V$. If $v=aw$
  with $a\in A$ and $|w|=|v|-1$, then $ea$ and $we$ belong to $J$ and,
  by Green's lemma, since $eawe\in H_e$, there is an idempotent
  $e_1\in J$ in the intersection of the \Cl L-class of $ea$ with the
  \Cl R-class of $we$ and so $eve=eae_1we$. Similarly, if $w=bu$ with
  $b\in A$ and $|u|=|w|-1$, then there is an idempotent $e_2\in J$
  such that $e_1we=e_1be_2ue$. Proceeding inductively in this manner,
  we conclude that $eve$ may be factorized as a product of elements of
  the form $e'ae''$ where $a\in A$ and $e'$ and $e''$ are idempotents
  of~$J$. In conclusion, finitely many elements of the special form
  $e'ae''$ ($a\in A$, $e',e''\in E(J)$) generate a profinite
  subsemigroup $T$ of $S_0$ that contains all maximal subgroups
  in~$J$. Since $S_0$ is locally finite, we deduce that $T$ is finite
  and, hence, the maximal subgroups in~$J$ are finite. Since $J$ has
  only finitely many idempotents, we also conclude that $J$ is finite.
\end{proof}

We may now improve Theorem~\ref{t:finmany-reg-J-classes} in the
special case of the pseudovarieties of Theorem~\ref{t:IE(H)-m-locfin}
as follows.

\begin{Cor}
  \label{c:IE-m-locfin}
  If $\pv H\subseteq\pv G$ and $\pv V\subseteq\pv S$ are locally
  finite pseudovarieties, then every pseudoword over \pv{(N\vee
    H)\malcev V} is an $\omega$-word of height at most~1.
\end{Cor}

\begin{proof}
  By Theorem~\ref{t:finmany-reg-J-classes}, it suffices to show that
  every idempotent $e$ of the semigroup $\Om A{\bigl((N\vee H)\malcev
    V\bigr)}$ is an $\omega$-word of height~1. Let $(u_n)_n$ be a
  sequence of words converging to~$e$. Then we have $\lim
  u_n^\omega=e$. But, since $\Om A{\bigl((N\vee H)\malcev V\bigr)}$
  has only finitely many idempotents, the convergent sequence
  $(u_n^\omega)_n$ must eventually stabilize at its limit $e$. Hence,
  $e$ is indeed an $\omega$-word of height~1.
\end{proof}

The special case where the pseudovariety \pv V of
Corollary~\ref{c:IE-m-locfin} is~\pv{Sl} leads to quite familiar
pseudovarieties. Indeed, it is easy to see that %
$\pv{(N\vee H)\malcev Sl}=\pv{DH}$ and, in particular, $\pv{N\malcev
  Sl}=\pv J$. Thus, Corollary~\ref{c:IE-m-locfin} generalizes the
first author's result that every pseudoword over~\pv J is an
$\omega$-word of height at most one~\cite{Almeida:1990b}.

In view of the results of this section, it is natural to ask whether a
countable finitely generated profinite semigroup is necessarily
finitely generated in the signature consisting of multiplication and
$\omega$-power. This is left as an open problem.

\section{Mal'cev product with locally finite pseudovarieties}
\label{sec:malcev-locfin}

The question addressed in this and the next sections is whether the
Mal'cev product of a locally countable pseudovariety with a locally
finite pseudovariety is also locally countable. We consider first the
special case where locally finite pseudovariety consists of nilpotent
semigroups.

\begin{Thm}
  \label{t:Malcev-N}
  Suppose that $S$ is a profinite semigroup and $\varphi:S\to N$ is a
  continuous homomorphism into a finite nilpotent semigroup $N$ such
  that $\varphi^{-1}(0)$ is locally countable. If $S$ is finitely
  generated then so is $\varphi^{-1}(0)$. In particular, $S$ is
  locally countable.
\end{Thm}

\begin{proof}
  Suppose that $S$ is generated, as a topological semigroup by a
  finite subset~$A$. Let $n$ be such that $N$ satisfies the identity
  $x_1\cdots x_n=0$. We claim that $\varphi^{-1}(0)$ is generated by
  the following set, where $A^k$ denotes the set of all words of
  length~$k$ on the alphabet $A$:
  $$B=\bigcup_{n\le k<2n}A^k
  \cup\left(\bigcup_{k<n}A^k\setminus\varphi^{-1}(0)\right).$$
  Note that the finite set $B$ is certainly contained
  in~$\varphi^{-1}(0)$ since products of length at least $n$ are zero
  in~$N$. On the other hand, given an arbitrary element $s$
  of~$\varphi^{-1}(0)$, since $\varphi^{-1}(0)$ is an open set and the
  subsemigroup $T$ of $S$ generated algebraically by $A$ is dense
  in~$S$, there is a sequence $(w_r)_r$ of elements
  of~$T\cap\varphi^{-1}(0)$ converging to~$s$. Each $w_r$ admits a
  factorization in the elements of~$A$, say in $k_r$ factors. If
  $k_r<n$ then $w_r$ belongs to the second term in the union defining
  $B$. Otherwise, $w_r$ admits a factorization into elements of the
  first term of the union defining $B$, simply by taking successive
  blocks of $n$ factors of~$A$ and combining the remainder block with
  the previous block. Hence, $s$ belongs to the closed subsemigroup
  generated by~$B$, which shows that $\varphi^{-1}(0)$ is finitely
  generated as a topological semigroup. Since $\varphi^{-1}(0)$ is
  locally countable by assumption, it follows that it is countable. To
  conclude that $S$ is countable it now suffices to apply
  Theorem~\ref{t:Brown-profinite}.
\end{proof}

In contrast, if $\varphi:G\to H$ is a continuous homomorphism from an
infinite finitely generated profinite group to a finite group, then
$\varphi^{-1}(1)$ is finitely generated as a topological group. To
prove it, it suffices to consider the case where $G$ is a free
profinite group. The preimage $\varphi^{-1}(1)$ is then a clopen
subgroup in which the intersection $I$ with the free group $F$ on the
same finite set of generators is dense. Since $I$ is a subgroup of
finite index of~$F$, $I$ is finitely generated by the Nielsen-Schreier
Theorem. Hence $\varphi^{-1}(1)$ is finitely generated as a
topological (semi)group. A more precise statement can be found
in~\cite[Section~3.6]{Ribes&Zalesskii:2010}.

\begin{Cor}
  \label{c:Malcev-N}
  Let \pv V be a locally countable pseudovariety and let \pv W be a
  locally finite pseudovariety of nilpotent semigroups. Then the
  pseudovariety $\pv V\malcev\pv W$ is locally countable.\qed
\end{Cor}

It is well known that the atoms in the lattice of all pseudovarieties
of semigroups are $\op xy=0\cl$, \pv{Sl}, the pseudovariety
$\pv{Ab}_p$ of all finite elementary Abelian $p$-groups, where $p$~is
prime, and the pseudovarieties \pv{LZ} and \pv{RZ}, respectively of
all finite left-zero and right-zero semigroups. Note that the largest
pseudovariety not containing all but the first one is precisely \pv N.
The next theorem shows that the pseudovariety \pv N is optimal in
Theorem~\ref{t:Malcev-N}.

\begin{Thm}
  \label{t:Malcev-non-N}
  For every non-nilpotent finite semigroup $T$, there is a non-locally
  countable profinite semigroup $S$ and a continuous homomorphism
  $\varphi:S\to U$ into a divisor of~$T$ such that $\varphi^{-1}(e)$
  is locally countable for every $e\in E(U)$.
\end{Thm}

\begin{proof}
  It suffices to show that, if \pv W is any of the atoms \pv{Sl},
  $\pv{Ab}_p$, \pv{LZ} and \pv{RZ}, then there is a locally countable
  pseudovariety \pv V such that $\pv V\malcev\pv W$ is not locally
  countable. This is proved in the next section. for all but the
  atom~\pv{RZ}, for which the result follows from the case of~\pv{LZ}
  by left/right duality.
\end{proof}

\begin{Cor}
  \label{c:Malcev-loc-fin}
  Let \pv W be a locally finite pseudovariety. Then $\pv V\malcev\pv
  W$ is locally countable for every locally countable pseudovariety
  \pv V if and only if \pv W is contained in~\pv N.\qed
\end{Cor}

\section{Mal'cev products with non-nilpotent atoms}
\label{sec:examples}

We prove in this section the claim stated in the proof of
Theorem~\ref{t:Malcev-non-N} that for every non-nilpotent atom \pv W
in the lattice of pseudovarieties of semigroups there is a locally
countable pseudovariety \pv V such that $\pv V\malcev\pv W$ is not
locally countable.

\subsection{The atom \texorpdfstring{\pv{Sl}}{Sl}}
\label{sec:Sl}

The following construction seems to have been first used in
\cite{Rhodes&Allen:1973} in the so-called synthesis theory and plays a
role in several contexts
\cite{Rhodes:1986a,Rhodes:1986,Diekert&Kufleitner&Weil:2011,
  Almeida&Klima:2011a,Almeida&Klima:2015a}. Here, we extend it to
profinite semigroups.

Let $S$ and $T$ be profinite semigroups and let $f:S\to T$ be a
continuous mapping. Consider the Cartesian product $K=S\times T\times
S$. We define on $U(S,T,f)=S\uplus K$ a multiplication extending that
of~$S$ as follows:
\begin{itemize}
\item $(s_1,t,s_2)(s_1',t',s_2')=(s_1,tf(s_2s_1')t',s_2')$;
\item $s(s_1,t,s_2)=(ss_1,t,s_2)$;
\item $(s_1,t,s_2)s=(s_1,t,s_2s)$.
\end{itemize}
It is easy to see that this multiplication on~$U(S,T,f)$ is
associative (this has been done in the purely algebraic setting for
instance in~\cite[Lemma~3.1]{Almeida&Klima:2011a}). If we endow
$U(S,T,f)$ with the coproduct topology of the space $S$ with the
product space $K$, then $U(S,T,f)$ is a compact 0-dimensional space in
which the multiplication is continuous. By a theorem of Numakura
\cite[Theorem~1]{Numakura:1957}, it follows that $U(S,T,f)$ is a
profinite semigroup.

\begin{Thm}
  \label{t:malcev-Sl}
  For every nontrivial locally finite pseudovariety of groups~\pv H,
  the join $\pv N\vee\pv{CS(H)}$ is locally countable but the Mal'cev
  product $(\pv N\vee\pv{CS(H)})\malcev\pv{Sl}$ is not.
\end{Thm}

\begin{proof}
  Since it is well known that $\pv{CS(H)}=\pv H*\pv{RZ}$, it follows
  from Theorems~\ref{t:star} and~\ref{t:join} that the join $\pv
  N\vee\pv{CS(H)}$ is locally countable. To prove that $(\pv
  N\vee\pv{CS(H)})\malcev\pv{Sl}$ is not locally countable, we exhibit
  an uncountable finitely generated profinite semigroup $U$ and a
  continuous homomorphism $\varphi:U\to F$ onto a finite semilattice
  $F$ such that, for each $e\in F$, the subsemigroup $\varphi^{-1}(e)$
  is pro-$(\pv N\vee\pv{CS(H)})$. The theorem then follows from
  Theorem~\ref{t:Malcev}.

  To construct $U$, consider the monoid $M=\mathbb{N}\cup\{\infty\}$
  under addition, where the topology is given by the one-point
  compactification of the discrete set~$\mathbb{N}$ of natural
  numbers. Let $G=\Om MH$ and let $f:M\to G$ be the natural generating
  function. We take $U=U(M,G,f)$ to be the profinite semigroup defined
  above.

  For $F$, we take the meet-semilattice given by the three-element
  chain $0\le 1\le 2$. The mapping $\varphi$ sends $0\in M$ to $2$,
  $M\setminus\{0\}$ to~$1$, and $K=M\times G\times M$ to~$0$. Note
  that $\varphi$ is a continuous homomorphism.
  We claim that $U$ is not locally countable.

  We first observe that $U$ is generated as a topological monoid by
  the two elements $a=1\in M$ and $b=(0,1,0)$ where the middle 1 is
  the idempotent of the group $G$. Indeed, $1\in M$ generates the
  topological monoid~$M$, while the elements $ba^ib=(0,f(i),0)$
  generate the topological group $\{0\}\times G\times\{0\}$ and
  $a^i(0,g,0)a^j=(i,g,j)$. Moreover, $U$ is uncountable since $G$ is
  isomorphic to the maximal subgroups in its minimum ideal and $G$ is
  an infinite profinite group (as the function $f$ is injective),
  whence uncountable, as observed in Section~\ref{sec:loc-count-pvs}.

  To prove the claim, it remains to show that $K$ is locally finite.
  This follows from showing that $K$ is pro-\pv{CS(H)}:
  the calculation
  $$(i,g,j)=(i,gh^{-1}f(k)^{-1},0)(k,h,\ell)(0,f(\ell)^{-1},j)$$
  shows that $K$ has only one \Cl J-class while the maximal subgroups
  of~$K$ are isomorphic with~$G$.
  
  To conclude the proof, it suffices to observe that
  $\varphi^{-1}(2)=\{1\}$, $\varphi^{-1}(1)\simeq\Om1N$, and
  $\varphi^{-1}(0)=K$ is pro-\pv{CS(H)}, as shown above. By
  Theorem~\ref{t:Malcev} it follows that the profinite monoid $U$ is
  pro-$(\pv N\vee\pv{CS(H)})\malcev\pv{Sl}$.
\end{proof}

The semigroup $U$ of the preceding proof may also be used to establish
the following result, which shows that the hypothesis that \pv W is
locally finite may not be dropped in Corollary~\ref{c:Malcev-N}.

\begin{Thm}
  \label{t:Malcev-N-bis}
  There is a locally finite pseudovariety \pv V such that $\pv
  V\malcev\pv N$ is not locally countable.
\end{Thm}

\begin{proof}
  Let $S$ be the profinite semigroup that is obtained from $U$ by
  removing the identity element. Note that, in the notation of the
  proof of Theorem~\ref{t:malcev-Sl}, $I=K\cup\{\infty\}$ is an ideal
  of $S$ and $S/I$ is isomorphic with \Om1N. Taking into account that
  $K$ is locally finite, it is easy to see that $I$ is also locally
  finite. By the proof of Theorem~\ref{t:malcev-Sl}, $S$ is not
  locally countable. Moreover, if \pv V is a pseudovariety such that
  $I$ is pro-\pv V, then $S$ is pro-$(\pv V\malcev\pv N)$ by
  Theorem~\ref{t:Malcev}. Thus, to complete the proof, it suffices to
  exhibit a locally finite pseudovariety \pv V for which $I$ is
  pro-\pv V. We claim that the locally finite pseudovariety
  $\pv{Sl}\vee\pv{CS(H)}$ has the required property.

  Consider the product $P=K\times\{0,1\}$ of $K$ with the two-element
  semilattice and the subset $T=(K\times\{0\})\cup\{(e,1)\}$, where
  $e=(\infty,f(\infty)^{-1},\infty)$. Since $e$ is an idempotent, $T$
  is a closed subsemigroup of~$P$. The profinite semigroup $P$ is
  pro-$(\pv{Sl}\vee\pv{CS(H)})$ since $K$ is pro-\pv{CS(H)}. To
  establish the claim, it thus suffices to show that $T$ is isomorphic
  with~$I$, which is a consequence of the following calculations: for
  $(x,g,y)$ in $K$, we have
  \begin{itemize}
  \item $e(x,g,y)
    =(\infty,f(\infty)^{-1}\cdot f(\infty)\cdot g,y)
    =(\infty,g,y)
    =\infty (x,g,y)$;
  \item dually, $(x,g,y)e=(x,g,y)\infty$.\popQED
  \end{itemize}
\end{proof}

\subsection{The atom \texorpdfstring{\pv{LZ}}{LZ}}
\label{sec:LZ}

Our treatment of the atom \pv{LZ} requires a more casuistic and
complicated construction than that of~\pv{Sl}. This subsection
provides a somewhat long and technical proof of the following result,
where the terminology is explained later.

\begin{Thm}
  \label{t:LZ}
  There is an aperiodic locally countable pseudovariety \pv U of
  semigroups of dot depth~1 such that $\pv U\malcev\pv{LZ}$ is not
  locally countable.
\end{Thm}

Let $k$ be a positive integer. We consider the semigroup $S_k$ with
zero given by the following presentation
\begin{align*}
  \langle\ a,b \mid\ 
  & a^{k+1}=a^{k},\  b^{k+1}=b^{k},\ a^{k} b^{k}
    a^{k}=a^{k},\
    b^{k} a^{k} b^{k}=b^{k}, \\
  & a^nb^na=b^na^nb=0\  (n< k)  \
    \rangle\, .
\end{align*}
The elements of $S_k$ may be represented by words of the form 
\begin{equation}
\label{eq:canonical-form}
w=a^{\gamma_0} b^{\gamma_1} \dots a^{\gamma_{2\ell}}b^{\gamma_{2\ell+1}}.
\end{equation}
To describe the unique representation of each element different
from~$0$, we first denote $E_k=\{0,1,\dots , k\}$ the set of potential
exponents, since we may assume that $\{\gamma_0,\gamma_1, \dots ,
\gamma_{2\ell+1}\}\subseteq E_k$ as higher powers may be reduced by
applying the rules $a^{k+1}\mapsto a^k$ and $b^{k+1}\mapsto b^k$. And,
of course, we assume that only $\gamma_0$ and $\gamma_{2\ell+1}$ may
take the value $0$. Then we consider the relation $\prec_k$ on $E_k$
given by the formula $i \prec_k j$ if $i<j$ or $i=j=k$. We may assume
that for each $0\le i<2\ell-1$ we have $\gamma_i \prec_k
\gamma_{i+1}$, since otherwise we have $w=0$ by the defining relations
$a^nb^na=b^na^nb=0$. For the same reason, we may assume that
$\gamma_{2\ell-1} \prec_k \gamma_{2\ell}$ in case
$\gamma_{2\ell+1}\not =0$. Furthermore, we may assume that at most two
exponents take the value $k$, since otherwise we could shorten the
word by applying the transformations $a^{k} b^{k} a^{k}\mapsto a^{k}$
or $b^{k} a^{k} b^{k}\mapsto b^{k}$. This describes canonical forms of
words in $S_k$. The canonical form may be obtained from each word by
applying rules which are oriented by the defining relations from
longer words to shorter ones. To say it more formally, we may add a
new symbol 0 and the rules $a0,0a,b0,0b \mapsto 0$, and mention that
the resulting rewriting system is obviously confluent with canonical
forms described above. Finally, notice that the semigroup $S_k$ is
finite, since there are only finitely many canonical words.

Now, let $\pv V$ be the pseudovariety generated by the set of all such
semigroups $S_k$ with $k\ge1$. Notice that all semigroups $S_k$ are
aperiodic and, therefore, we have $\pv V\subseteq \pv A$. More
precisely, $S_k$ is a semigroup of ``dot depth~1''. The
Straubing-Th\'erien dot depth hierarchy is a filtration of~\pv A as an
infinite increasing chain of pseudovarieties, which has been
extensively studied. The natural definition of the hierarchy comes
from language theory (cf.~\cite{Pin:1997}) via Eilenberg's
correspondence between pseudovarieties of semigroups and so-called
``varieties of regular languages'' \cite{Eilenberg:1976}. The second
level of the hierarchy (that starts at level~0, given by the trivial
pseudovariety \pv I) is known as dot depth~1 and has been shown by
Knast \cite{Knast:1983a} to be defined by the pseudoidentity
\begin{equation}
  \label{eq:Knast}
  (exfye)^\omega xft(ezfte)^\omega=(exfye)^\omega(ezfte)^\omega
\end{equation}
where $e=u^\omega$, $f=v^\omega$ and $t,u,v,x,y,z$ are distinct
variables.

\begin{Lemma}
  \label{l:dd1}
  The semigroup $S_k$ is of dot depth~1.
\end{Lemma}

\begin{proof}
  All nonzero idempotents of~$S_k$ lie in the same \Cl D-class $D$,
  which consists of all elements in canonical
  form~\eqref{eq:canonical-form} for which the first or second nonzero
  exponent is~$k$.
  It follows that, if a product of the form $esf$ is a factor of a
  nonzero idempotent, where $e$ and $f$ are idempotents, then $esf$
  belongs to~$D$. Thus, $esf$ is the only element in the intersection
  of the \Cl R-class of $e$ with the \Cl L-class of $f$. Hence, if
  $fte$ is also a factor of a nonzero idempotent, then the equality
  $esfte=e$ follows from Green's Lemma
  (cf.~\cite[Proposition~2.3.7]{Howie:1995}) and aperiodicity. This
  shows that $S_k$ satisfies the pseudoidentity~\eqref{eq:Knast}.
\end{proof}

The next step for the proof of Theorem~\ref{t:LZ} is the following
result.

\begin{Lemma}
  \label{l:malcev-lz-uncountable}
  The pseudovariety $\pv V$ is not locally countable.
\end{Lemma}

\begin{proof}
  We fix the alphabet $A=\{a,b\}$ and consider the relatively free
  profinite semigroups $\Om A A$ and $\Om A V$. We denote $\eta : \Om
  A A \rightarrow \Om A V$ and $\psi_k : \Om A V \rightarrow S_k$
  ($k\ge1$) the natural continuous homomorphisms; all these
  homomorphisms map the generating set $A$ identically. We describe an
  uncountable subset $X$ in $\Om A A$ such that the restriction of
  $\eta$ to $X$ is an injective mapping.

  Let $s=(s_i)_{i\in \mathbb N}$ be an increasing sequence of natural
  numbers. For each such $s$, we consider the following sequence of
  words: we put $w_1=a^{s_1}b^{s_2}$ and, for each $i>1$, we let
  $w_{i}=w_{i-1}a^{s_{2i-1}}b^{s_{2i}}$. We claim that the sequence
  $(w_i)_{i\in\mathbb N}$ converges in $\Om A A$. Indeed, let $\alpha
  : \Om A A \rightarrow S$ be a continuous homomorphism into a finite
  aperiodic semigroup~$S$. Then, for some $n$, we have $u^n=u^{n+1}$
  for every element $u\in S$. In particular, we have
  $\alpha(a)^m=\alpha(a)^n$ and $\alpha(b)^m=\alpha(b)^n$ for every
  $m>n$. Thus, for such $m>n$, we have $\alpha(w_m)=\alpha\bigl(w_n
  (a^nb^n)^{m-n}\bigr)$ and therefore $\alpha(w_m)=\alpha\bigl(w_n
  (a^nb^n)^{n}\bigr)$ for every $m\ge 2n$ and we see that the sequence
  $(\alpha(w_i))_{i\in\mathbb N}$ is eventually constant. Let $w_s$ be
  the limit in~\Om AA of the converging sequence $(w_i)_{i\in\mathbb
    N}$.
  Let $X$ be the set of all $w_s$'s obtained in this way.

  Let $s$ and $t$ be distinct increasing sequences of natural numbers.
  Let $j$ be the minimum index such that $s_j\ne t_j$. Without loss of
  generality, we may assume that $s_j<t_j$. Now we put $k=s_j+1$ and
  consider the semigroup $S_k$. Let us assume first that $j$ is an
  even index. We see that, for $i$ such that $2i>j$, we have $\psi_k
  (\eta(w_i))= a^{s_1}b^{s_2}\dots a^{s_{j-1}} b^{s_j} a^k b^k$. Let
  $w'_i=a^{t_1}b^{t_2} \dots a^{t_{2i-1}}b^{t_{2i}}$ be the words in
  the given sequence converging to $w_t$. Since $t_j\ge k$, we get
  that, for every $i$ such that $2i>j$, the equality
  $\psi_k(\eta(w'_i))= a^{s_1}b^{s_2}\dots a^{s_{j-1}} b^{k}$ holds.
  Therefore, the element $\psi_k (\eta(w_s))$ of~$S_k$ is equal to
  $a^{s_1}b^{s_2}\dots a^{s_{j-1}} b^{s_j} a^k b^k$, while $\psi_k
  (\eta(w_t))$ is equal to $a^{s_1}b^{s_2}\dots a^{s_{j-1}} b^{k}$.
  Thus, $\psi_k (\eta(w_s))\not=\psi_k(\eta(w_t))$ in the case of even
  $j$. The case of odd $j$ may be treated in the same way. Hence, we
  have $\eta(w_s)\not=\eta(w_t)$ and we conclude that $X$ is
  uncountable and the restriction of~$\eta$ to the set~$X$ is an
  injective mapping, as claimed.
\end{proof}

For each $k$, let $T_k$ be the subsemigroup of $S_k$ consisting of $0$
and the elements in canonical form~\eqref{eq:canonical-form} with
$\gamma_0\not=0$. Let $\pv U$ be the pseudovariety generated by all
finite semigroups $T_k$ ($k\ge2$). Thus, we have $\pv U\subseteq \pv V
\subseteq \pv A$.

Let $C$ be a finite set. It is known that, for arbitrary $u\in\Om C A$
and every factor $x\in C$ of~$u$, there are $u',u''\in (\Om CA)^1$ such
that $u=u'xu''$ and $x$ is not a factor of~$u'$, where $u'$ is
uniquely determined (see
\cite{Almeida&Trotter:1999a}).\footnote{Actually, by
  ``equidivisibility'' of the pseudovariety \pv A, $u''$ is also
  unique. See \cite{Almeida&ACosta:2017}, where equidivisible
  pseudovarieties are characterized.} We talk about the \emph{first
  occurrence} of $x$ in $u$. We denote by $\flat^x_y(u)$ the
\emph{number of occurrences} of an element $y$ of~$C$ in $u'$: if it
exists, $\flat^x_y(u)$ is the maximum non-negative number~$n$ such
that $y^n$ is a subword of~$u'$; otherwise, $\flat^x_y(u)$ is the
symbol~$\infty$. Thus, $\flat^x_y(u)$ is defined for any pair of
distinct letters $x,y$ such that $x$ occurs in~$u$. Dually, if we
consider the last occurrence of $x$ in $u$ then we get $u=v'xv''$
where $x$ does not appear in the uniquely determined $v''$; we
denote by $\sharp^x_y(w)$ the number of occurrences of $y$ in~$v''$.
Furthermore, for each pair of not necessarily distinct letters $x$,
$y$ we consider the set $M_u(x,y)$ of all words $w\in C^*$ in which
$x$ and $y$ do not appear and such that $xwy$ is a factor of $u$. By
Higman's theorem \cite{Higman:1952}, for every set of words the set of
its minimal members in the subword ordering is finite. We denote this
finite subset of $M_u(x,y)$ by $m_u(x,y)$. Notice that $m_u(x,y)$ is
empty if and only if $u$ has no finite factor of the form $xwy$; also,
$m_u(x,y)=\{1\}$ if and only if $xy$ is a factor of $u$.

The following technical lemma plays a crucial role in our
investigation of the pseudovariety \pv U.

\begin{Lemma}
  \label{l:malcev-finite-factors}
  Let $u,v\in \Om C A$ be a pair of infinite pseudowords satisfying
  the following assumptions:
  \begin{enumerate}[(i)]
  \item\label{item:l:malcev-finite-factors-1} Both $u$ and $v$
    contain the word $x^3$ as a subword for every $x\in C$.
  \item\label{item:l:malcev-finite-factors-2} For every pair $x,y\in
    C$ of distinct letters, we have $\flat^x_y(u)=\flat^x_y(v)$ and
    $\sharp^x_y(u)=\sharp^x_y(v)$.
  \item\label{item:l:malcev-finite-factors-3} For each $x,y\in C$ we
    have $m_u(x,y)=m_v(x,y)$.
  \end{enumerate}
  Then $\pv U\models u=v$.
\end{Lemma}
  
\begin{proof}
  Let $\varphi :C \rightarrow T_k$ be an arbitrary mapping with
  $k\ge2$; we denote by the same symbol the unique extension to a
  continuous homomorphism $\varphi: \Om C A\rightarrow T_k$. We want
  to show that $\varphi(u)=\varphi(v)$. Let us assume for a
  contradiction that $\varphi(u)\not=\varphi(v)$.

  If $\varphi (x)=0$ for some $x\in C$ then $\varphi(u)=\varphi(v)=0$.
  Thus, we may assume that, for each $x\in C$, the image $\varphi(x)$
  is a canonical word of the form~\eqref{eq:canonical-form} with
  $\gamma_0\not=0$.
  Assume for a moment that some exponent $\gamma_{2i+1}$ of $b$ in
  $\varphi(x)$ is smaller than $k$. Since $x^3$ is a subword of $u$,
  we may factorize $u$ in the following way $u=u_0 x u_1 x u_2 x u_3$
  for some pseudowords $u_0,\dots ,u_3\in (\Om C A)^1$. Since the
  canonical forms of both $\varphi(u_1 x)$ and $\varphi(u_2 x)$ start
  with $a$, a representation of $\varphi(x u_1 x u_2 x)$ in the
  form~\eqref{eq:canonical-form} contains two factors of the form
  $ab^{\gamma_{2i+1}}a$, which means that $\varphi(u)=0$. So, we reach
  a contradiction and from hereon we may assume that $\varphi(x)$ is
  of the form $a^{\gamma_0}b^ka^{\gamma_{2}}$ or it is a power of $a$.
  Denote by $X$ the set of all $x\in C$ such that $\varphi(x)$ is a
  power of $a$. If all letters are mapped to powers of $a$, then we see
  that $\varphi(u)=\varphi(v)=a^k$ since both $u$ and $v$ are infinite
  pseudowords. Hence, we may assume that there is $y\in C\setminus X$
  such that there is a factorization $u=u_0yu'$ with $u_0\in\Om XA$.
  By~(\ref{item:l:malcev-finite-factors-1}), there is also a
  factorization $v=v_0yv'$ where $y$ is not a factor of~$v_0$.
  By~(\ref{item:l:malcev-finite-factors-2}), the factors of~$v_0$
  from~$C$ are the same as those in~$u_0$. So, $y$~is also the
  leftmost factor of~$v$ from~$C\setminus X$. Taking also into account
  the left/right dual argument, we conclude that we may assume that
  there are letters $y,z\in C\setminus X$ and factorizations
  $u=u_0yu_1zu_2$ and $v_0yv_1zv_2$ such that $u_0,v_0,u_2,v_2\in (\Om
  X A)^1$ are uniquely determined. By the
  assumption~(\ref{item:l:malcev-finite-factors-2}), we get that
  $\varphi(u_0)=\varphi(v_0)=a^m$ and $\varphi(u_2)=\varphi(v_2)=a^n$
  for some natural numbers $m$ and $n$. Then there are just two
  possible values of $\varphi(u)$ and $\varphi(v)$, namely $0$ and
  $a^{m+\gamma_0}b^k a^j$ with $j=k$ or $j=\gamma_2+n$ where
  $\gamma_0$ is the exponent of the first $a$ in $\varphi(y)$ and
  $\gamma_2$ is the exponent of the last $a$ in $\varphi(z)$.

  Now, since the images $\varphi(u)$ and $\varphi(v)$ are distinct, we
  see that exactly one of them is equal to~$0$. We may assume without
  loss of generality that it is $\varphi(u)$. Let $(u_n)_n$ be a
  sequence of words from $C^+$ converging to~$u$ in~\Om CA. Since
  $\varphi$~is continuous, without loss of generality we may assume
  that $\varphi(u_n)=0$ for every $n$. Since every letter
  in~$C\setminus X$ maps to an element of~$T_k$ of the form
  $a^{\gamma_0}b^ka^{\gamma_2}$ and those in $X$ map to powers of~$a$,
  we deduce that there is a finite factor $y'w'z'$ of $u_n$ such that
  $\varphi(y')=a^{\alpha_0}b^ka^{\alpha_2}$, $\varphi(z')=a^{\beta_0}
  b^k a^{\beta_2}$, $\varphi(w')=a^m$ (where we let $a^0=1$), and
  $\alpha_2+m+\beta_0<k$. As there are only finitely many
  possibilities for such factors $y'w'z'$, for $|w'|<k$, there are
  infinitely many values of~$n$ for which we may choose the same such
  factor $y'w'z'$. Hence, there is also such a factor $y'w'z'$ of~$u$.
  Using the hypothesis~(\ref{item:l:malcev-finite-factors-3}) for the
  pair of letters $y',z'$, we know that there is a word $w''$ which is
  a subword of $w'$ and such that $y'w''z'$ is a factor of $v$.
  Clearly, $\varphi(w'')=a^n$ with $n\le m$ and, therefore, we have
  $\varphi(v)=0$, which contradicts the assumption at the beginning of
  the paragraph.
\end{proof}

We are now ready to establish the following result, which is part of
the proof of Theorem~\ref{t:LZ}.

\begin{Prop}
  \label{p:malcev-countable-U}
  The pseudovariety $\pv U$ is locally countable.
\end{Prop}

\begin{proof}
  Let $C$ be a finite set and consider $\Om C U$ and $\Om C A$. We
  prove inductively with respect to the size of $C$ that $\Om C U$ is
  countable. Since $\pv U\subseteq \pv A$, we see that $\Om C U$ is
  countable for every singleton set $C$. We next assume that $\Om B U$
  is countable for every proper subset $B$ of $C$. We show that there
  is a countable subset $W$ of $\Om C A$ such that for each $u\in \Om
  C A$ there is $w\in W$ such that $\pv U \models u=w$. This proves
  that the cardinality of $\Om C U$ is at most the cardinality of the
  set $W$, whence \Om CU is countable.

  At first, assume that $u$ does not contain all words from $C^*$ as
  subwords and consider some shortest word $v$ that is not a subword
  of~$u$. Then, there is a factorization $v=a_1a_2\dots a_na_{n+1}$
  where $a_1a_2\dots a_n$ is a subword of~$u$. Moreover, there is a
  factorization of $u$ of the form $u=u_1a_1u_2a_2\dots a_nu_{n+1}$
  such that each letter $a_i$ does not occur in $u_i$ for $1\le i \le
  n+1$. Therefore, each $u_i$ belongs to $\Om{B}A$ for some proper
  subset $B$ of~$C$. By induction hypothesis, there is a countable set
  $W_B\subseteq \Om{B}A$ of representatives of elements of $\Om{B}U$.
  We construct the set $\overline{W}$ of all pseudowords of the form
  $w_0a_1w_1a_2\dots a_nw_{n+1}$ with $w_i\in\Om{B_i}A$, where
  $B_i=C\setminus\{a_i\}$ ($i=1,\ldots,n+1$). The set $\overline W$ is
  countable and for each pseudoword $u$ which does not contain all
  words from $C^*$ as subwords there is a pseudoword
  $w\in\overline W$ such that $\pv U\models u=w$.

  Thus, it remains to deal with the case when $u$ contains all words
  from $C^*$ as subwords. Denote by $Z$ the set of all such
  pseudowords $u$. We consider an equivalence relation $\sim$ on $Z$
  such that $u\sim v$ if $u$ and $v$ satisfy the condition
  (\ref{item:l:malcev-finite-factors-2})
  and~(\ref{item:l:malcev-finite-factors-3}) of
  Lemma~\ref{l:malcev-finite-factors}; note that the definition of $Z$
  entails that its elements are infinite and that
  condition~(\ref{item:l:malcev-finite-factors-1}) of
  Lemma~\ref{l:malcev-finite-factors} is satisfied for every pair of
  elements of~$Z$. Thus, the relation $\sim$ is in fact an
  intersection of finitely many equivalence relations on $Z$ each of
  which has countably many classes. Hence, also $\sim$ has countably
  many classes and we may choose from each class one representative
  and collect them into the countable subset $W'$ of $Z\subseteq \Om C
  A$. By Lemma~\ref{l:malcev-finite-factors}, we deduce that for each
  $u \in Z$ there is $w\in W'$ such that $\pv U\models u=w$. We have
  proved that there is a countable set $W=\overline W \cup W'$ with
  the required property.
\end{proof}

\begin{Prop}
  \label{p:malcev-LZ}
  The pseudovariety $\pv U\malcev \pv{LZ}$ is not locally countable.
\end{Prop}

\begin{proof}
  Let $S$ be the two-element set $\{a,b\}$ and consider the operation
  on $S$ given by $xy=x$. Thus, $S$ is a left zero semigroup. For
  every $k\ge2$, we denote by $R_k$ the subsemigroup of $S_k\times S$,
  which is generated by the pair of elements $(a,a)$ and $(b,b)$.
  Clearly, both projections $\pi_1: S_k\times S \rightarrow S_k$ and
  $\pi_2: S_k\times S \rightarrow S$ to the first and second
  coordinates are surjective when restricted to $R_k$. Note that
  $\pi_2^{-1}(a)=T_k\times\{a\}\in \pv U$. Also, the subsemigroup
  $\pi_2^{-1}(b)$ is isomorphic to $T_k$ by applying the automorphism
  of $S_k$ that exchanges the generators $a$ and $b$. Hence $R_k$
  belongs to $\pv U \malcev \pv{LZ}$ and since $S_k$ is a homomorphic
  image of $R_k$, we get $\pv V\subseteq \pv U \malcev \pv{LZ}$. By
  Lemma~\ref{l:malcev-lz-uncountable}, it follows that $\pv U\malcev
  \pv{LZ}$ is not locally countable.
\end{proof}

\begin{proof}[Proof of Theorem~\ref{t:LZ}]
  The result follows from Propositions~\ref{p:malcev-countable-U}
  and~\ref{p:malcev-LZ} together with Lemma~\ref{l:dd1}.
\end{proof}

\subsection{The atom \texorpdfstring{$\pv{Ab}_p$}{Abp}}
\label{sec:Abp}

Throughout this subsection, $p$ denotes a fixed prime number. Our aim
is to establish the following result.

\begin{Thm}
  \label{t:Abp}
  There is a locally countable pseudovariety $\pv U_p$ such that the
  Mal'cev product $\pv U_p\malcev\pv{Ab}_p$ is not locally countable.
\end{Thm}

The proof follows the same lines as the proof of Theorem~\ref{t:LZ},
based on a similar construction. We include just those details in
which the new construction differs from the previous one.

Let $k$ be a positive integer. We consider the semigroup $S_k(p)$ with
zero given by the following presentation
$$
\langle\ a,b \mid a^{2}=0,\ b^{k+1}=b^{k},\ b^{k} (a b^{k})^p =b^{k},\
b^n a b^na=0\ ( n< k ) \ \rangle\, .$$ %
Since $a^2=0$, the elements of $S_k(p)$ different from~$0$ may be
represented by words of the form
\begin{equation}
  \label{eq:canonical-form-group}
  w=b^{\beta_0} a b^{\beta_1} a {b^{\beta_2}} a \dots a b^{\beta_{\ell}}.
\end{equation}
Here, the integer $\ell$ denotes the number of occurrences of the
letter $a$ in the word, including the case $\ell=0$ where the
considered word is $b^{\beta_0}$ with a positive exponent $\beta_0$.
To describe canonical forms of words in $S_k(p)$, we use the same
notation $E_k$ and $i \prec_k j$ as in Subsection~\ref{sec:LZ}. So, we
assume that $\{\beta_0, \dots , \beta_{\ell}\}\subseteq E_k$ since
$b^{k+1}= b^k$, and only $\beta_0$ and $\beta_{\ell}$ may take the
value $0$, where $b^0$ is interpreted as the empty word. Furthermore,
for $0\le i<\ell-1$ we assume that $\beta_i \prec_k \beta_{i+1}$, and
at most $p$ of the exponents $\beta$ take the value $k$, since
otherwise we could shorten the word by applying the equality $b^{k} (a
b^k)^p = b^k$. In particular, under all these assumptions, for
canonical words of the form~(\ref{eq:canonical-form-group}) we have
$\ell \le k+p$ and we conclude that there are only finitely many
canonical forms representing elements of the semigroup $S_k(p)$, which
is consequently finite. Moreover, for $w$ in the canonical
form~(\ref{eq:canonical-form-group}) with $\ell\not=0$ and such that
$w^3\not=0$, we see that $\beta_1=\dots=\beta_{\ell-1}=k\le
\beta_0+\beta_\ell$. Then, for each $m\ge 3$, we have
$w^{m}=b^{\beta_0}(ab^k)^{m\ell-1}ab^{\beta_\ell}$. Thus
$w^{p+3}=w^3$, an equality that holds also in the case where $w^3=0$.
If $w$ is the canonical word~(\ref{eq:canonical-form-group}) with
$\ell=0$, then $w=b^{\beta_0}$ and we have $w^{k+1}=w^k$. Hence
$S_k(p)$ satisfies the identity $x^{\max\{k,3\}+p}=x^{\max\{k,3\}}$.

Let $\pv V_p$ be the pseudovariety generated by the semigroups
$S_k(p)$ with $k\ge2$. By the arguments at the end of the previous
paragraph we obtain the inclusion $\pv V_p\subseteq \op
x^{\omega+p}=x^\omega \cl$. We denote by $\pv W_p$ the latter
pseudovariety which consists of all finite semigroups containing as
nontrivial subgroups only groups of exponent $p$. Notice that $\pv
A\subseteq \pv W_p$ for an arbitrary prime $p$.

\begin{Lemma}
  \label{l:malcev-groups-uncountable}
  For each prime $p$, the pseudovariety $\pv V_p$ is not locally
  countable.
\end{Lemma}

\begin{proof}
  Although the proof is a quite straightforward modification of the
  proof of Lemma~\ref{l:malcev-lz-uncountable}, we include the details
  since the canonical forms of elements in $S_k(p)$ are slightly
  different from those in $S_k$.

  Let $\eta : \Om{A}W_p \rightarrow \Om{A}V_p$ and $\psi_k : \Om{A}V_p
  \rightarrow S_k(p)$ ($k\ge2$) be the natural continuous
  homomorphisms which map the generating set $A=\{a,b\}$ identically.
  Our aim is to show that $\Om{A}V_p$ is not countable.

  For an increasing sequence of natural numbers $s=(s_i)_{i\ge1}$, we
  consider the following sequence of words: $w_1=ab^{ps_1}ab^{ps_2}
  \dots ab^{ps_p}$ and for each $i>1$ we put
  $$w_{i}=w_{i-1} \, a b^{p s_{p(i-1)+1}}\,
  a b^{p s_{p(i-1)+2}} \dots a b^{p s_{pi}}\, .$$ %
  We show that the sequence $(w_i)_{i\ge1}$ converges in $\Om A {\pv
    W_p}$. Let $S$ be an arbitrary finite semigroup from $\pv W_p$ and
  let $\alpha : \Om AW_p \rightarrow S$ be a continuous homomorphism.
  There is an integer $n$ such that $s^n=s^{n+p}$ for every element
  $s\in S$. Clearly, if we fix such $n$ (depending just on the
  semigroup $S$), we also have $s^m=s^{m+p}$ for every $m>n$. In
  particular, for every $m>n$, $s^{pm}=s^{pn}$ is a unique idempotent
  which is a power of $s\in S$. To simplify a notation, we denote
  $a_S=\alpha(a)$ and $b_S=\alpha(b)$. Then we have
  $\alpha(ab^{ps_m})=a_Sb_S^{ps_m}=a_Sb_S^{pn}$ for every $m>n$
  because $s_m\ge m$. Now, for each $m\ge 2n$, we deduce that
  $\alpha(w_m)=\alpha(w_n) (a_S b_S^{pn})^{p(m-n)}=\alpha(w_n) (a_S
  b_S^{pn})^{pn}$. Hence, the sequence $(\alpha(w_i))_{i\ge1}$ is
  eventually constant and $(w_i)_{i\ge1}$ converges. Let $w_s$ be the
  limit of the converging sequence $(w_i)_{i\ge1}$.

  We show that $\eta (w_s)\not =\eta (w_t)$ for every pair of distinct
  increasing sequences of natural numbers $s\not=t$. This proves that
  $\Om{A}V_p$ is uncountable as there are uncountably many increasing
  sequences of natural numbers.

  Assume that $(w_i)_{i\ge 1}$ and $(w'_i)_{i\ge 1}$ are the
  constructed converging sequence for $s$ and $t$ respectively. Let
  $j$ be the minimum index such that $s_j\ne t_j$ and assume that
  $s_j<t_j$ as the opposite case can be treated in the same way. We
  put $k=p(s_j+1)$ and consider the semigroup $S_k(p)$. Denote by $j'$
  the unique integer such that $p(j'-1) < j \le pj'$. For each $i\ge
  j'$, we have
  $$\psi_k (\eta(w_i)) %
  = a b^{ps_1} a b^{ps_2} \dots ab^{ps_{j-1}} a b^{ps_j} (a %
  b^k)^{pj'-j} (a b^k )^{p(i-j')}.$$ %
  Thus we deduce that $\psi_k (\eta(w_s))$ is equal to $ab^{ps_1}\dots
  ab^{ps_{j-1}} a b^{ps_j} (a b^k )^{\ell}$, for a certain
  $\ell\in\{1,\dots p\}$. On the other hand, the word %
  $w'_i=a b^{pt_1} a b^{pt_2}\dots a b^{pt_{pi}} $ is such that, for
  every $i\ge j'$, the equality
  $$\psi_k (\eta(w'_i))
  =a b^{ps_1} a b^{ps_2} \dots ab^{ps_{j-1}} a b^{k} (a b^k )^{pj'-j}
  (a b^k )^{p(i-j')}$$ %
  holds. Thus, $\psi_k (\eta(w_t))$ is equal to $a b^{ps_1} \dots
  ab^{ps_{j-1}} (a b^k )^{\ell+1}$ with the same $\ell$ as above. We
  conclude that $\psi_k (\eta(w_s))\not=\psi_k(\eta(w_t))$ in $S_k(p)$
  which implies that $\eta(w_s)\not=\eta(w_t)$.
\end{proof}

For each $k$, we denote $T_k(p)$ the subsemigroup of $S_k(p)$
consisting of the element $0$ and all elements in canonical
form~\eqref{eq:canonical-form-group} where the number of occurring
$a$'s, that is $\ell$, is divisible by $p$. The pseudovariety
generated by all semigroups $T_k(p)$ is denoted $\pv U_p$. Notice that
$\pv U_p\subseteq \pv V_p \subseteq \pv W_p$.

Our aim is the same as in the previous subsection, namely to prove
inductively with respect to the size of a finite set $C$ that $\Om C
U_p$ is countable. First of all, for $u\in\Om C W_p$, we may define
all technical notions, that is $\flat^x_y(u)$, $\sharp^x_y(w)$,
$M_u(x,y)$ and $m_u(x,y)$, in the same way, because $\pv W_p$
satisfies the same conditions as $\pv A$ which enable us to apply the
results from~\cite{Almeida&Trotter:1999a} when we define the first
occurrence of the letter in the pseudoword $u$.

We formulate a slight modification of
Lemma~\ref{l:malcev-finite-factors}.

\begin{Lemma}
  \label{l:malcev-group-finite-factors}
  Let $u,v\in \Om C W_p$ be a pair of infinite pseudowords satisfying
  the following assumptions:
  \begin{enumerate}[(i)]
  \item\label{item:l:malcev-group-finite-factors-1} For every $x\in
    C$, both $u$ and $v$ contain $x^2$ as a subword.
  \item\label{item:l:malcev-group-finite-factors-2} For every pair $x,y\in
    C$ of distinct letters, we have $\flat^x_y(u)=\flat^x_y(v)$ and
    $\sharp^x_y(u)=\sharp^x_y(v)$.
  \item\label{item:l:malcev-group-finite-factors-3} For each $x,y\in
    C$, we have $m_u(x,y)=m_v(x,y)$.
  \end{enumerate} Then $\pv U_p\models u=v$.
\end{Lemma}

\begin{proof}
  We proceed in the same way as in the proof of
  Lemma~\ref{l:malcev-finite-factors}. Let $\varphi: \Om C
  W_p\rightarrow T_k(p)$ be a continuous homomorphism and assume for a
  contradiction that $\varphi(u)\not=\varphi(v)$.

  The case when $\varphi (x)=0$ for some $x\in C$ is clear. So, assume
  that, for each $x\in C$, the element $\varphi(x)\in T_k(p)$ is
  represented by the canonical word $\varphi(x)=b^{\beta_0} a
  b^{\beta_1} \dots a b^{\beta_{\ell}}$ with $\ell$ divisible by $p$.
  If there is $1\le i< \ell$ such that $\beta_i <k$, then
  $\varphi(u)=\varphi(v)=0$ since $x^2$ is a subword of both $u$ and
  $v$. Thus we may also assume that every $\varphi(x)$ is of the form
  $b^{\beta_0}(ab^k)^{p-1} ab^{\beta_{p}}$ or $b^{\beta_0}$. Now we
  denote $X$ the set of all $x\in C$ such that $\varphi(x)$ is a power
  of $b$ and the rest of the proof is essentially the same as in the
  case of Lemma~\ref{l:malcev-finite-factors} and is omitted.
\end{proof}

\begin{Prop}
  \label{p:malcev-countable-U-group}
  For each prime $p$, the pseudovariety $\pv U_p$ is locally
  countable.
\end{Prop}

\begin{proof}
  This result can be proved in the same way as
  Proposition~\ref{p:malcev-countable-U}. The pseudovariety $\pv A$ is
  replaced by $\pv W_p$ and the pseudovariety $\pv U$ by $\pv U_p$. At
  the basis of induction, we again have that $\Om C W_p$ is countable
  whenever $C$~is a singleton set.
\end{proof}

\begin{Prop}
  \label{p:malcev-group}
  For every prime $p$, the pseudovariety $\pv U_p\malcev \pv{Ab}_p$ is
  not locally countable.
\end{Prop}

\begin{proof}
  Let $G=\{1,a,a^2,\dots, a^{p-1}\}$ be a cyclic group of order~$p$.
  For every $k\ge2$, we consider the semigroup $S_k(p)\times G$ and
  its subsemigroup $R_k$ generated by $(a,a)$ and $(b,1)$. Let $\pi_2
  : R_k \rightarrow G$ be the restriction of the projection from
  $S_k(p)\times G$ onto the second coordinate. Then
  $\pi_2^{-1}(1)=T_k(p)\in \pv U_p$, where $1$ is a unique idempotent
  in the group $G$. Thus $R_k$ belongs to $\pv U_p \malcev \pv{Ab_p}$.
  Considering the restriction of the projection from $S_k(p)\times G$
  onto the first coordinate, we see that $S_k(p)$ is a homomorphic
  image of $R_k$. Thus, we have $\pv V_p\subseteq \pv U_p \malcev
  \pv{Ab}_p$. We deduce that $\pv U_p\malcev \pv{Ab}_p$ is not locally
  countable by Proposition~\ref{l:malcev-groups-uncountable}.
\end{proof}

\begin{proof}[Proof of Theorem~\ref{t:Abp}]
  It suffices to invoke Propositions \ref{p:malcev-countable-U-group}
  and~\ref{p:malcev-group}.
\end{proof}

Unlike the case of the atom \pv{LZ}, we do not have a familiar upper
bound for the pseudovariety $\pv U_p$. Yet, the same argument as in
the proof of Lemma~\ref{l:dd1} shows that $S_k(p)$ satisfies the
pseudoidentity obtained from~\eqref{eq:Knast} by raising both sides to
the $p$th power.

\section{Pseudovarieties of aperiodic inverse semigroups}
\label{sec:A-ESl}

Recall that a semigroup is said to be an \emph{inverse semigroup} if,
for every element $s$ there is a unique element $t$ such that $sts=s$
and $tst=t$; an element $t$ satisfying these equalities is called an
\emph{inverse} of~$s$. The existence of an inverse characterizes
regular elements. An inverse semigroup may also be characterized as a
semigroup in which every element is regular and in which idempotents
commute. Another important property of inverse semigroups is that they
arise precisely as semigroups of partial bijections, that is, of
bijections between two subsets of a fixed set; the operation is
composition and the set should be closed under taking function
inverses.

Denote by \pv{Inv} the class of all finite inverse semigroups. The
pseudovariety of semigroups generated by~\pv{Inv} turns out to be
\pv{ESl} \cite{Ash:1987}, which may me decomposed both as $\pv{Sl}*\pv
G$ and as $\pv{Sl}\malcev\pv G$. It is natural to ask whether a
pseudovariety bound on the groups in a class of finite inverse
semigroups leads to a smaller bound on a pseudovariety containing the
class. The negative answer to this question may be found
in~\cite[Theorem~5.3]{Higgins&Margolis:1999}: in particular, if
$\pv{Inv}\cap\pv A\subseteq\pv{DA}\malcev\pv H$ for a pseudovariety of
groups \pv H, then $\pv H=\pv G$.

The main purpose of this section is to show that the pseudovariety
$\pv{ESl}\cap\pv A$ is not locally countable. The proof of this result
requires some combinatorial tools which we proceed to introduce.

Given a word $w=a_1a_2\cdots a_n$ over an alphabet $A$, we consider
the following linear $A$-labeled digraph $\Lambda_w$:
$$q_1\xrightarrow{a_1}q_2\xrightarrow{a_2}q_3\xrightarrow{}\cdots
\xrightarrow{a_n}q_{n+1}.$$ The vertices $q_1$ and $q_{n+1}$ are
called respectively the \emph{beginning} and \emph{end} vertices. For
words $w_1,\ldots,w_r$, the \emph{flower digraph}
$\Phi(w_1,\ldots,w_r)$ is obtained by taking the disjoint union of the
digraphs $\Lambda_{w_i}$ ($i=1,\ldots,r$) and identifying all
beginning and end vertices to a single vertex which is denoted $q_0$.

Consider the Prouhet-Thue-Morse endomorphism $\mu$ of the free monoid
$\{a,b\}^*$, which is defined by $\mu(a)=ab$ and $\mu(b)=ba$. Noting
that $\mu^n(a)$ is a prefix of $\mu^{n+1}(a)$, we see that the
sequence of words $\bigl(\mu^n(a)\bigr)_n$ determines an infinite
word, denoted $\mathbf{t}$, which is also known as the
\emph{Prouhet-Thue-Morse word}: the word $\mu^n(a)$ is the prefix
of~$\mathbf{t}$ of length $2^n$. See~\cite{Allouche&Shallit:1999} for
the name attribution and many relevant properties. In particular,
Thue~\cite{Thue:1912} proved that $\mathbf{t}$ is \emph{cube-free} and
even \emph{overlap-free} in the sense that it contains respectively no
factor of the forms $w^3$ and $uvuvu$ with $u,w\in\{a,b\}^+$ and
$v\in\{a,b\}^*$. The latter corresponds to a minimal situation where
we find overlapping occurrences of the same word. By direct
inspection, one verifies that the only words in $\{a,b\}^+$ of length
at most~3 that are not factors of~$\mathbf{t}$ are the cubes $a^3$ and
$b^3$. Note also that if the finite word $w$ is a factor
of~$\mathbf{t}$ then so is $\mu(w)$.

For $n\ge0$, we let $\Gamma_n=\Phi\bigl(\mu^n(a),\mu^n(b)\bigr)$ where
$\mu^0$ means the identity mapping. The vertex $q_0$ of the graph
$\Gamma_n$ is denoted $0_n$. The first three $\{a,b\}$-labeled
digraphs $\Gamma_n$ are depicted in Figure~\ref{fig:Gamma-1,2}.
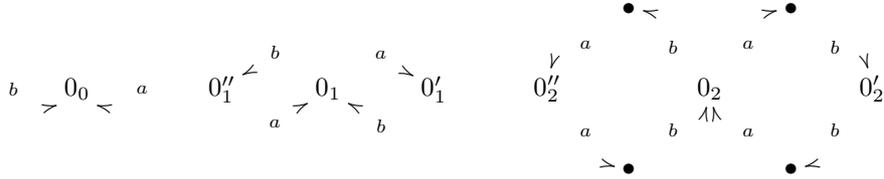
\begin{figure}[ht]
  \centering\small
  $\vcenter{\hbox{
      \newdir{|>>}{*:(1,+0.2)@^{>} *:(1,+0.2)@_{>}}
      \newdir{>>>}{*:(1,-0.2)@^{>} *:(1,-0.2)@_{>}} \xymatrix@d{ 0_0
        \ar@{-|>>}@(ru,rd)[]^a \ar@{->>>}@(lu,ld)[]_b }}}$
  \quad
  $\vcenter{\hbox{ \xymatrix{
        0_1'' %
        \ar@/_3mm/[r]_a & 0_1 %
        \ar@/^3mm/[r]^a \ar@/_3mm/[l]_b & 0_1' %
        \ar@/^3mm/[l]^b }}}$
  \qquad
  $\vcenter{\hbox{ \xymatrix@C=6mm@R=6mm{
        & \bullet %
        \ar@/_3mm/[ld]^a
        && \bullet %
        \ar@/^3mm/[rd]_b
        & \\ 0_2'' \ar@/_3mm/[rd]^a
        && 0_2 %
        \ar@/^3mm/[ru]_a \ar@/_3mm/[lu]^b
        && 0_2' \ar@/^3mm/[ld]_b \\
        &
        \bullet %
        \ar@/_3mm/[ru]^b
        && \bullet \ar@/^3mm/[lu]_a
        & }}}$
  \caption{The flower digraphs $\Gamma_0$, $\Gamma_1$ and $\Gamma_2$}
  \label{fig:Gamma-1,2}
\end{figure}

For an $A$-labeled digraph $\Delta$, one may view each letter $a\in A$
as representing a binary relation on the set of vertices, namely the
relation $\Delta(a)$ consisting of the pairs of vertices $(p,q)$ such
that there is an edge $p\xrightarrow{a}q$ in the digraph. The monoid
of binary relations generated by all $\Delta(a)$ with $a\in A$ is
called the \emph{transition} monoid of~$\Delta$ and is denoted
$T(\Delta)$. Note that, for a word $w=a_1a_2\cdots a_n$ ($a_i\in A$),
a pair of vertices $(p,q)$ belongs to the product
$\Delta(a_1)\Delta(a_2)\cdots\Delta(a_n)$ if and only if there is a
path in~$\Delta$ beginning at $p$ and ending at $q$ for which $w$ is
the concatenation of the labels of the successive edges; in this case,
we also write $p\xrightarrow{w}q$. If the relations $\Delta(a)$ ($a\in
A$) are partial functions, then every word $w$ in~$A^*$ induces a
right action on the vertices of $\Delta$ by $p\cdot w=q$ if
$p\xrightarrow{w}q$. For simplicity, we then also say that words of
$A^*$ \emph{act} on~$\Delta$.

In case distinct edges either starting or ending at the same vertex
have distinct labels, the relations $\Delta(a)$ are partial
bijections. Since the words $\mu^n(a)$ and $\mu^n(b)$ start with
different letters and end with different letters, that property holds
for the flower digraphs $\Gamma_n$. We denote by $M_n$ the transition
monoid of the disjoint union
$\tilde{\Gamma}_n=\biguplus_{i=0}^n\Gamma_i$. Since it is generated by
partial bijections of a finite set, $M_n$ embeds in the inverse monoid
of all partial bijections of the vertex set of~$\tilde{\Gamma}_n$ and
so $M_n$ belongs to the pseudovariety~$\pv{ESl}$. Note that the
restriction of the mappings from $M_{n+1}$ to the graph
$\tilde{\Gamma}_n$ defines an onto homomorphism $\varphi_n:M_{n+1}\to
M_n$. Also note that, if we let $T_n=T(\Gamma_n)$, then $M_n$ is a
subdirect product of the monoids $T_0,\ldots,T_n$. The fact that the
identity mapping on generators does not extend to a homomorphism
$T_{n+1}\to T_n$ explains the need to work with $M_n$ to obtain an
inverse system. For instance, $\mu^n(a^3)$ acts as a nonempty
transformation on~$\Gamma_n$ while, as the arguments in the proof of
Lemma~\ref{l:x4=x3} show, it acts as the empty transformation
on~$\Gamma_{n+1}$.

Recall that a language $L\subseteq A^*$ over a finite alphabet $A$ is
said to be \emph{star-free} if it may be expressed in terms of the
languages $\emptyset$, $\{1\}$ and $\{a\}$ ($a\in A$) using only the
operations of finite union, finite product, and complement (in~$A^*$)
\cite{Pin:1997}.

There is a nice alternative description of the semigroup $T_n$ which
we proceed to give. Consider $\Gamma_n$ as a (deterministic,
incomplete) automaton $\Cl A_n$ with $0_n$ as the only initial and
terminal state. The language it recognizes is simply the image
$\image\mu^n$ of the homomorphism $\mu^n$. Since no quotient of the
automaton recognizes the same language, $T_n$ is the syntactic monoid
of $\image\mu^n$. We claim that the language $\image\mu^n$ is
star-free. We prove this below by induction on $n\ge0$. The induction
step depends on the following result.

\begin{Lemma}
  \label{l:1substitution}
  Let $\eta$ be an injective endomorphism of the free monoid $A^*$
  such that $\image\eta$ is star-free. If $L\subseteq A^*$ is a
  star-free language then so is $\eta(L)$.
\end{Lemma}

\begin{proof}
  By assumption, $L$ is obtained from the languages $\emptyset$,
  $\{1\}$, and $\{a\}$ ($a\in A$) by applying a finite number of times
  the operations of finite union, finite product and complement in
  $A^*$. Proceeding by induction on the number of times the operations
  are applied, $\eta(L)$ may be obtained from the languages
  $\emptyset$, $\{1\}$, and $\{\eta(a)\}$ ($a\in A$) by applying the
  operations of finite union, finite product, and
  $\eta(K)\mapsto\eta(A^*\setminus K)$. Since finite languages are
  star-free and homomorphisms preserve union and multiplication, all
  we need to show is that, in the induction process, the image under
  $\eta$ of the complement also produces star-free languages. Indeed,
  since $\eta$ is injective, we obtain the following equalities:
  $$\eta(A^*\setminus K) %
  =\{\eta(w):w\in A^*\setminus K\} %
  =(\image\eta)\cap(A^*\setminus\eta(K)).$$ %
  By the induction hypothesis, we may assume that $\eta(K)$ is
  star-free. Hence, so is $\eta(A^*\setminus K)$ as $\image\eta$ is
  assumed to be star-free.
\end{proof}

It is well known and easy to check that syntactic monoids of star-free
languages are aperiodic. The converse is a key result of
Schützenberger \cite{Schutzenberger:1965}.

We may now easily prove our claim.

\begin{Prop}
  \label{p:starfree}
  The language $\image\mu^n$ is star-free.
\end{Prop}

\begin{proof}
  We first note that an elementary hand (or computer)
  calculation shows that $T_1$ is a 15-element aperiodic inverse
  monoid. Hence, by Schützenberger's theorem, the language
  $\image\mu=\{ab,ba\}^*$ is star-free. On the other hand, the
  language $\image\mu^0=\{a,b\}^*$ is star-free, being the complement
  of the empty language. Also note that $\mu$ is injective. Finally,
  as $\image\mu^{n+1}=\mu(\image\mu^n)$, assuming inductively that
  $\image\mu^n$ is star-free, Lemma~\ref{l:1substitution} yields that
  so is $\image\mu^{n+1}$.
\end{proof}

As an immediate consequence of Proposition~\ref{p:starfree}, we obtain
the following result.

\begin{Cor}
  \label{c:Tn-in-A}
  The monoid $T_n$ is aperiodic.\qed
\end{Cor}

The following is a little technical observation about the action of
certain words on~$\Gamma_n$.

\begin{Lemma}
  \label{l:rank-1-identities}
  The two words $\mu^{n+1}(a)$ and $\mu^{n+1}(b)$ act in the same way
  on~$\Gamma_n$, namely as the identity at the single vertex $0_n$.
\end{Lemma}

\begin{proof}
  To prove the lemma, note that, by construction of the
  graph~$\Gamma_n$, the words $\mu^n(a)$ and $\mu^n(b)$ label circuits
  at the vertex $0_n$. Since $\mu^{n+1}(a)=\mu^n(a)\mu^n(b)$, this
  word also labels a circuit at the vertex $0_n$. We need to show that
  it does not label a path starting at any other vertex of~$\Gamma_n$.
  Indeed, otherwise, there would be some word $u=u_1u_2u_3$
  ($u_i\in\{a,b\}$) and a factorization
  $$\mu^n(u)=x\mu^{n+1}(a)y=x\mu^n(a)\mu^n(b)y$$
  with $1\le|x|<2^n$. In particular, $\mu^n(a)$ is a factor
  of~$\mu^n(u_1u_2)$ which overlaps with $\mu^n(u_2)$ and $\mu^n(b)$
  is a factor of $\mu^n(u_2u_3)$ which also overlaps
  with~$\mu^n(u_2)$. Since both $\mu^n(u_1u_2)$ and $\mu^n(u_2u_3)$
  are factors of~$\mathbf{t}$ and no two consecutive occurrences of a
  factor of~$\mathbf{t}$ may overlap, by considering respectively the
  factor $\mu^n(a)$ of $\mu^n(u_1u_2)$ and the factor $\mu^n(b)$
  of~$\mu^n(u_2u_3)$ we conclude that $u_2$ may be neither $a$ nor
  $b$, which is absurd. Interchanging the roles of $a$ and~$b$, we
  conclude that also $\mu^{n+1}(b)$ acts on the vertices of~$\Gamma_n$
  precisely as the identity on the vertex $0_n$.
\end{proof}

There is a natural graph homomorphism respecting labels
$\gamma_n:\Gamma_{n+1}\to\Gamma_n$: it maps $0_{n+1}$ to~$0_n$ and
each vertex $0_{n+1}\cdot w$ to $0_n\cdot w$; this mapping is well
defined since the two simple cycles at $0_{n+1}$ in $\Gamma_{n+1}$ are
labeled by the words $\mu^{n+1}(a)$ and $\mu^{n+1}(b)$, which
fix~$0_n$.
The preimage of $0_n$ in $\Gamma_{n+1}$ consists of three vertices,
namely $0_{n+1}$, $0_{n+1}'=0_{n+1}\cdot\mu^n(a)$ and
$0_{n+1}''=0_{n+1}\cdot\mu^n(b)$. The two vertices $0_{n+1}'$ and
$0_{n+1}''$ are distinguished in particular by the fact that from the
first only leaves an arrow labeled $b$ (which is the first letter
of~$\mu^n(b)$) while from the second only leaves an arrow labeled $a$
(the first letter of~$\mu^n(a)$).

To simplify the proof and as it is sufficient for our purposes, the
following lemma states only a special case of a much more general
phenomenon.

\begin{Lemma}
  \label{l:lifting}
  Given any transformation $t\in T(\Gamma_n)$ of domain $0_n$, there
  are words $u$ and $v$ whose action on~$\Gamma_n$ coincides with~$t$
  and whose actions on~$\Gamma_{n+1}$ are distinct transformations of
  domain~$0_{n+1}$.
\end{Lemma}

\begin{proof}
  Suppose that $t$ is given by $t:0_n\to p$. Let $w$ be the label of
  the shortest path from $0_n$ to~$p$. Since $w$ is a prefix of
  $\mu^n(c)$ for some $c\in\{a,b\}$ (there is a choice for $c$ only in
  case $w=1$), we consider such a letter $c$. Let $u=\mu^{n+2}(a)w$
  and $v=\mu^{n+2}(a)\mu^n(d)w$, where $d$ is the letter such that
  $\{c,d\}=\{a,b\}$. By Lemma~\ref{l:rank-1-identities}, the domain of
  the actions of $u$ and $v$ on~$\Gamma_{n+1}$ is reduced to~$0_{n+1}$
  and they both act like $t$ on~$\Gamma_n$.
\end{proof}

The following result was discovered by computer calculation for small
values of~$n$. The proof that it holds for all $n\ge0$ is somewhat
technical.

\begin{Lemma}
  \label{l:x4=x3}
  The monoid $M_n$ satisfies the identity $x^4=x^3$.
\end{Lemma}

\begin{proof}
  It suffices to show that each monoid $T_n$ ($n\ge1$) satisfies the
  identity $x^4=x^3$.
  
  We claim that if $w\in\{a,b\}^+$ then either $w^3=w^2$ in~$T_n$ or,
  whenever $p\xrightarrow{w^r}q$ in~$\Gamma_n$ for some $r\ge 3$ then
  $p=q$ and $p\xrightarrow wp$, so that $p\xrightarrow{w^k}q$ for
  every $k\ge1$. Indeed, having a path from $p$ to $q$ labeled $w^r$
  implies that $w^r$ is a factor of $\mu^n(u)$ for some word $u$ that
  we may assume to be of minimum length. More precisely, there is a
  factorization $\mu^n(u)=xw^ry$ with $0\le\max\{|x|,|y|\}<2^n$ and
  $0_n\cdot x=p$ (cf.~Figure~\ref{fig:wr},
  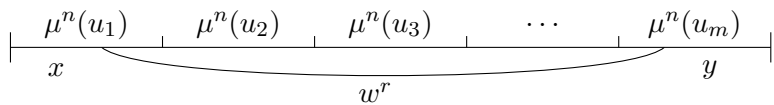
\begin{figure}[ht]
    \begin{center}
      \begin{tikzpicture}[x=1mm, y=1mm]
        \draw (0,5)--(100,5); %
        \draw (0,3)--(0,7); %
        \draw (100,3)--(100,7); %
        \foreach \x in {20,40,60,80}%
           \draw (\x,5)--(\x,6.5); %
        \draw (12,5) .. controls (18,0) and (80,0) .. (86,5); %
        \draw (10,8) node {$\mu^n(u_1)$}; %
        \draw (30,8) node {$\mu^n(u_2)$}; %
        \draw (50,8) node {$\mu^n(u_3)$}; %
        \draw (70,8) node {$\ldots$}; %
        \draw (90,8) node {$\mu^n(u_m)$}; %
        \draw (48,-1) node {$w^r$}; %
        \draw (6,2) node {$x$}; %
        \draw (92,2) node {$y$};
      \end{tikzpicture}
    \end{center}
    \caption{A factorization involving $w^r$}
    \label{fig:wr}
  \end{figure}
  where $u=u_1u_2\cdots u_m$,
  with the $u_i\in\{a,b\}$).
  Since $r\ge3$ and $\mathbf{t}$ is
  cube-free, $\mu^n(u)$ cannot be a factor of~$\mathbf{t}$ and,
  therefore, neither may be $u$. Hence, we must have $|u|\ge3$ and at
  least one of the words $\mu^n(a)$ or $\mu^n(b)$ must be a factor
  of~$w^r$. Since the two situations are symmetric, we assume that
  $\mu^n(a)$ is a factor of~$w^r$.
  Say by Lemma~\ref{l:rank-1-identities}, $\mu^n(a)$ acts as the
  identity at the single vertex $0_{n-1}$ of~$\Gamma_{n-1}$. We deduce
  that the domain of the action of $\mu^n(a)$ on~$\Gamma_n$ consists
  of the vertices $0_n$ and $0_n''$. Moreover, by definition of the
  graph $\Gamma_n$, $\mu^n(a)$ fixes both $0_n$ and $0_n''$.

  If either $\mu^n(a)$ or $\mu^n(b)$ is actually a factor of~$w$, then
  the domain of the (partial injective) action of~$w$ on~$\Gamma_n$
  has at most two vertices. Since $p\cdot w^3$ is defined, either
  $p\cdot w=p$, and we are done, or $p\cdot w\ne p$ and $p\cdot
  w^2=p$. The latter case is excluded because $T_n$ is aperiodic by
  Corollary~\ref{c:Tn-in-A}. Hence, we may assume that neither
  $\mu^n(a)$ nor $\mu^n(b)$ is a factor of~$w$.

  Suppose now that $w$ is not a factor of either $\mu^n(a)$ or
  $\mu^n(b)$. In this case, the beginning of the factorization of
  $\mu^n(u)$ must be as in Figure~\ref{fig:www-mu}.
  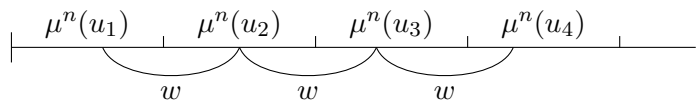
\begin{figure}[ht]
    \begin{center}
      \begin{tikzpicture}[x=1mm, y=1mm]
        \draw (0,5)--(90,5); %
        \draw (0,3)--(0,7); %
        \foreach \x in {20,40,60,80} %
        \draw (\x,5)--(\x,6.5); %
        \foreach \x in {12,30,48} %
           \draw (\x,5) .. controls (\x+2,0) and (\x+16,0) .. (\x+18,5); %
        \draw (10,8) node {$\mu^n(u_1)$}; %
        \draw (30,8) node {$\mu^n(u_2)$}; %
        \draw (50,8) node {$\mu^n(u_3)$}; %
        \draw (70,8) node {$\mu^n(u_4)$}; %
        \foreach \x in {21,39,57}
           \draw (\x,-1) node {$w$}; %
      \end{tikzpicture}
    \end{center}
    \caption{A factorization involving $w^3$}
    \label{fig:www-mu}
  \end{figure}
  Again, since $\mathbf{t}$ is cube-free, the word $u_1u_2u_3u_4$
  cannot be a factor of~$\mathbf{t}$ and so either the first three or
  the last three letters of~$u_1u_2u_3u_4$ must be equal. In either
  case, there is a letter $c$ such that $w^2$ appears as a factor of
  $\mu^n(c^3)$ as in Figure~\ref{fig:ww-mu}.
  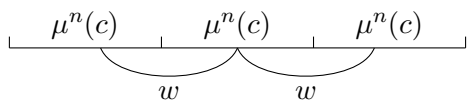
\begin{figure}[ht]
    \begin{center}
      \begin{tikzpicture}[x=1mm, y=1mm]
        \draw (0,5)--(60,5); %
        \foreach \x in {0,20,40,60} %
           \draw (\x,5)--(\x,6.5); %
        \foreach \x in {12,30} %
           \draw (\x,5) .. controls (\x+2,0) 
               and (\x+16,0) .. (\x+18,5); %
        \draw (10,8) node {$\mu^n(c)$}; %
        \draw (30,8) node {$\mu^n(c)$}; %
        \draw (50,8) node {$\mu^n(c)$}; %
        \foreach \x in {21,39}
           \draw (\x,-1) node {$w$}; %
      \end{tikzpicture}
    \end{center}
    \caption{A factorization involving $w^2$}
    \label{fig:ww-mu}
  \end{figure}
  If $|w|\ne2^n$, as in Figure~\ref{fig:ww-mu}, then comparing the two
  factorizations of $\mu^n(c^2)$, we find an overlap of $w$ within
  $\mathbf{t}$ (see Figure~\ref{fig:w-overlap}), which is impossible.
  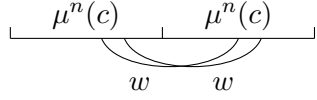
\begin{figure}[ht]
    \begin{center}
      \begin{tikzpicture}[x=1mm, y=1mm]
        \draw (0,5)--(40,5); %
        \foreach \x in {0,20,40} %
           \draw (\x,5)--(\x,6.5); %
        \draw (10,8) node {$\mu^n(c)$}; %
        \draw (30,8) node {$\mu^n(c)$}; %
        \foreach \x in {12,15} %
           \draw (\x,5) .. controls (\x+2,0)
                 and (\x+16,0) .. (\x+18,5); %
        \foreach \x in {17,28} %
           \draw (\x,-1) node {$w$}; %
      \end{tikzpicture}
    \end{center}
    \caption{A factorization involving $\mu^n(c^2)$}
    \label{fig:w-overlap}
  \end{figure}
  Hence, we must have $|w|=2^n$ and $w$ and $\mu^n(c)$ are
  conjugate words, that is, there is a factorization
  $\mu^n(c)=\alpha\beta$ such that $w=\beta\alpha$. It follows that
  $w^3=\beta(\alpha\beta)^2\alpha=w^2$ in~$T_n$ as $\mu^n(c)$ acts as
  a local identity on $\Gamma_n$, and this fulfills our claim.

  It remains to consider the case where in the factorization of
  Figure~\ref{fig:wr}, there is at least one of the factors $w$ in
  $w^r$ that falls completely and properly within one of the factors
  $\mu^n(u_i)$. Looking at such a factor $\mu^n(u_i)$, within which
  appears a full factor $w$ in the factorization of
  Figure~\ref{fig:wr}, we find a word $z=z_1z_2z_3$ (with
  $z_i\in\{a,b\}$) such that $w^3$ is a factor of~$\mu^n(z)$. Since
  the word $w^3$ is not a factor of~$\mathbf{t}$, the infinite word
  $\mathbf{t}$ cannot have $z$ as a factor. Hence, we must find $w^3$
  as a factor of~$\mu^n(c^3)$ for some letter $c$, but not as a factor
  of~$\mu^n(c^2)$. Depending on where each factor $w$ starts, the
  picture may look different but one possible configuration is
  represented in Figure~\ref{fig:www-mu-short}.
  \begin{figure}[ht]
    \begin{center}
      \begin{tikzpicture}[x=1mm, y=1mm]
        \draw (0,5)--(60,5); %
        \foreach \x in {0,20,40,60} %
           \draw (\x,5)--(\x,6.5); %
        \foreach \x in {8,22,36} %
           \draw (\x,5) .. controls (\x+2,0) 
                 and (\x+12,0) .. (\x+14,5); %
        \foreach \x in {0,20,40} %
           \draw (\x+10,8) node {$\mu^n(c)$}; %
        \foreach \x in {15,29,43}
           \draw (\x,-1) node {$w$}; %
      \end{tikzpicture}
    \end{center}
    \caption{Another factorization involving $w^3$}
    \label{fig:www-mu-short}
  \end{figure}
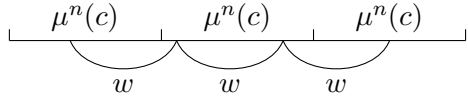
  In any case, comparing again the two factorizations of~$\mu^n(c^2)$,
  since $|w|<2^n$, we find again an overlap within $\mu^n(c^2)$ unless
  $|w|=2^{n-1}$ (see Figure~\ref{fig:w-overlap-short}), which is
  impossible. In the exceptional case, looking at the middle factor
  $\mu^n(c)$, we see that there is a factorization $\mu^n(c)=uwv$,
  where $u$ is a suffix of~$w$ and $v$~is a prefix of~$w$; since
  $|\mu^n(c)|=2|w|$, it follows that $w=vu$ and so
  $(uv)^2=\mu^n(c)=\mu^{n-1}(c)\mu^{n-1}(d)$, where
  $\{c,d\}=\{a,b\}$, which is impossible since the factors
  $\mu^{n-1}(c)$ and $\mu^{n-1}(d)$ have the same length and start
  with distinct letters.
  \begin{figure}[t]
    \begin{center}
      \begin{tikzpicture}[x=1mm, y=1mm]
        \draw (0,5)--(40,5); %
        \foreach \x in {0,20,40} %
           \draw (\x,5)--(\x,6.5); %
        \draw (10,8) node {$\mu^n(c)$}; %
        \draw (30,8) node {$\mu^n(c)$}; %
        \foreach \x in {2,8} %
           \draw (\x,5) .. controls (\x+2,0) 
                 and (\x+26,0) .. (\x+28,5); %
        \foreach \x in {12,28} %
           \draw (\x,-1) node {$w^2$}; %
        \end{tikzpicture}
    \end{center}
    \caption{Another factorization involving $\mu^n(c^2)$}
    \label{fig:w-overlap-short}
  \end{figure}
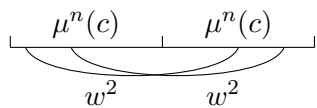
  This completes the proof of the lemma.
\end{proof}

Since $M_n$ is defined as a monoid of binary relations, it is ordered
by the relation $s\le t$ if $s\supseteq t$ in the sense that $\le$ is
a partial order that is stable under multiplication.

The pseudovariety of ordered monoids $\op1\le x^n\cl$ is the main
theme of the authors' paper \cite{Almeida&Klima:2017b}. The main
problem considered in that paper is to determine the pseudovariety of
monoids (or semigroups) $\langle\op1\le x^n\cl\rangle$ generated by
the ordered monoids in $\op1\le x^n\cl$ once the order is forgotten.
The conjectured result is the subpseudovariety of~\pv{BG} defined by
the pseudoidentities $(xy^n)^\omega = (y^nx)^\omega$ and
$x^{\omega+n}=x^\omega$, which is denoted $(\pv{BG})_n$.

\begin{Lemma}
  \label{l:1LEQxn}
  The ordered monoid $M_n$ satisfies the inequality $1\le x^k$ for
  every $k\ge3$.
\end{Lemma}

\begin{proof}
  Note that, since idempotents of~$M_n$ are partial bijections, they
  are identity mappings of some subset of the vertex set
  of~$\tilde{\Gamma}_n$. Thus, $M_n$ satisfies the inequality $1\le
  x^\omega$. By Lemma~\ref{l:x4=x3}, $M_n$ satisfies the identity
  $x^\omega=x^k$ for every $k\ge3$ and whence also the inequality
  $1\le x^k$.
\end{proof}

The homomorphisms $\varphi_n:M_{n+1}\to M_n$ constitute a chain and
determine an inverse limit $\varprojlim M_n$.
Our key result of this section is the following.

\begin{Thm}
  \label{t:limproj-Mn-uncountable}
  The inverse limit $\varprojlim M_n$ is uncountable.
\end{Thm}

\begin{proof}
  We may build a tree by taking as vertices at level $n$ the elements
  of~$M_n$ and letting the sons of a vertex $s\in M_n$ be the elements
  of $\varphi_n^{-1}(s)$. The elements of $\varprojlim M_n$ may then
  be identified with the infinite simple paths from the root. The
  result will follow from showing that the complete infinite binary
  tree embeds in our tree. The existence of such an embedding follows
  from Lemma~\ref{l:lifting}, which provides two alternatives for
  lifting elements of~$M_n$ whose action on $\Gamma_n$ is a
  transformation with domain $0_n$ to transformations whose action on
  each $\Gamma_{n+1}$ has domain~$0_{n+1}$.
\end{proof}

We conclude this section with consequences of
Theorem~\ref{t:limproj-Mn-uncountable}.

\begin{Cor}
  \label{c:several-non-loccount}
  None of the pseudovarieties \pv{A\cap ESl}, \pv{BV} (for every
  pseudovariety \pv V), $\langle\op1\le x^n\cl\rangle$, and
  $(\pv{BG})_n$ ($n\ge3$) is locally countable.
\end{Cor}

\begin{proof}
  By Lemma~\ref{l:x4=x3}, the monoid $M_n$ is aperiodic and its
  idempotents commute. Hence, the semigroup $S=\varprojlim
  M_n\setminus\{1\}$ is a pro-\pv{(A\cap ESl)} semigroup on two
  generators. Since $S$ is uncountable by
  Theorem~\ref{t:limproj-Mn-uncountable}, so is $\Om2{(A\cap ESl)}$.
  For \pv{BV}, the result follows from the inclusions $\pv{A\cap ESl}
  \subseteq\pv{BI}\subseteq\pv{BV}$. For the remaining
  pseudovarieties, the result follows from the inclusion
  $\langle\op1\le x^n\cl\rangle\subseteq(\pv{BG})_n$
  \cite[Proposition~4.2]{Almeida&Klima:2017b}, Lemma~\ref{l:1LEQxn}
  and Theorem~\ref{t:limproj-Mn-uncountable}.
\end{proof}

The questions as to whether $(\pv{BG})_n$ or $\pv B\op x^n=1\cl$ are
locally countable were raised in our paper
\cite[Section~8]{Almeida&Klima:2017b}. With the restriction $n\ge3$,
Corollary~\ref{c:several-non-loccount} does not handle the case of
$(\pv{BG})_2$ and $\langle\op1\le x^2\cl\rangle$, while
$(\pv{BG})_1=\langle\op 1\le x\cl\rangle=\pv J$ is locally countable.
We will not go into details, but the case of $n=2$ may be treated
similarly by working with a suitable substitution generating an
infinite square-free word instead of the Prouhet-Thue-Morse
substitution. One such substitution over a three letter alphabet is
given by $\varphi(a)=abc$, $\varphi(b)=ac$, $\varphi(c)=b$
\cite[Proposition~2.3.2]{Lothaire:1983}. The idea is to start with the
flower digraph
$\Phi\bigl(\varphi^n(a),\varphi^n(b),\varphi^n(c)\bigr)$ and apply the
Stallings folding procedure (see,
\cite{Stallings:1983,Margolis&Sapir&Weil:1999,Kapovich&Myasnikov:2002})
to reduce it to a labeled digraph $\Lambda_n$ in which each letter
acts as a partial bijection. The graphs $\Lambda_n$ for $n=1,\ldots,5$
are drawn in Figure~\ref{fig:sf}. The transition semigroup of
$\Lambda_n$ plays the role of the semigroup $T_n$ in the above
argument. Then, one may prove analogs of the previous results in this
section, which yield an extension of
Corollary~\ref{c:several-non-loccount} to cover the case $n=2$. We
leave open the question as to whether the profinite semigroups
$\Om2{}\langle\op 1\le x^2\cl\rangle$ and $\Om2{(BG)}_2$ are
countable.

\begin{figure}[ht] \centering
  
  \begin{tikzpicture}[->,>=stealth',shorten >=1pt,auto,
    node distance=1.3cm,
    initial/.style={initial above,initial text=,initial distance=5mm},
    accepting/.style={accepting above,accepting distance=5mm}]
    \footnotesize
    \node (1) {1};
    \node (0) [left of=1,initial,accepting] {0};
    \path (1) edge [loop right] node {$b$}
          (1) edge [bend right] node [above] {$c$} (0)
          (0) edge [loop left] node {$b$}
          (0) edge [bend right] node [below] {$a$} (1);
  \end{tikzpicture} \qquad
  \begin{tikzpicture}[->,>=stealth',shorten >=1pt,auto,
    node distance=1.3cm,
    initial/.style={initial above,initial text=,initial distance=5mm},
    accepting/.style={accepting above,accepting distance=5mm}]
    \footnotesize
    \node (1) {1};
    \node (2) [right of=1] {2};
    \node (3) [above of=2] {3};
    \node (0) [left of=3,initial,accepting] {0};
    \path (1) edge node {$b$} (2) edge [bend left] node {$c$} (0)
          (2) edge [bend left] node {$c$} (3)
          (3) edge [bend left] node {$a$} (2) edge node {$b$} (0)
          (0) edge [bend left] node {$a$} (1);
  \end{tikzpicture}
  \qquad
  \begin{tikzpicture}[->,>=stealth',shorten >=1pt,auto,
    node distance=1.3cm,
    initial/.style={initial above,initial text=,initial distance=5mm},
    accepting/.style={accepting above,accepting distance=5mm}]
    \footnotesize
    \node (1) {1};
    \node (2) [below of=1] {2};
    \node (3) [right of=2] {3};
    \node (4) [right of=3] {4};
    \node (5) [right of=4] {5};
    \node (6) [above of=5] {6};
    \node (7) [left of=6] {7};
    \node (0) [left of=7, initial,accepting]{0};
    \path (1) edge node {$b$} (2)
          (2) edge node {$c$} (3)
          (3) edge node {$a$} (4) edge node {$b$} (0)
          (4) edge node {$c$} (5)
          (5) edge node {$b$} (6)
          (6) edge node {$a$} (7)
          (7) edge node {$b$} (4)
          (7) edge node {$c$} (0)
          (0) edge node {$a$} (1);
  \end{tikzpicture}

  \bigskip

  \begin{tikzpicture}[->,>=stealth',shorten >=1pt,auto,
    node distance=1.3cm,
    initial/.style={initial above,initial text=,initial distance=5mm},
    accepting/.style={accepting above,accepting distance=5mm}]
    \footnotesize
    \def \n {8}
    \def \radius {1.5cm}
    \foreach \s in {1,...,7}
    {
      \node (\s) at ({360/(2*\n) + 360/\n * \s}:\radius) {$\s$};
    }
    \node (0) [initial,accepting] at ({360/(2*\n) + 360/\n * \n}:\radius) {$0$};
    \begin{scope}[xshift=4cm]
      \foreach \s in {8,...,15}
      {
        \node (\s) at ({180 + 360/(2*\n) + 360/\n * \s}:\radius)
        {$\s$};
      }
    \end{scope}
    \path (0) edge node {$a$} (1)
          (1) edge node {$b$} (2)
          (2) edge node {$c$} (3)
          (3) edge node {$a$} (4)
          (4) edge node {$c$} (5)
          (5) edge node {$b$} (6)
          (6) edge node {$a$} (7)
          (7) edge node {$b$} (8) edge node {$c$} (0)
          (8) edge node {$c$} (9)
          (9) edge node {$b$} (10)
          (10) edge node {$a$} (11)
          (11) edge node {$c$} (12)
          (12) edge node {$a$} (13)
          (13) edge node {$b$} (14)
          (14) edge node {$c$} (15)
          (15) edge node {$b$} (0) edge node {$a$} (8);
  \end{tikzpicture}

  \bigskip

  \begin{tikzpicture}[->,>=stealth',shorten >=1pt,auto,scale=1.1,
    initial/.style={initial above,initial text=,initial distance=5mm},
    accepting/.style={accepting above,accepting distance=5mm}]
    \footnotesize
    \def \n {16}
    \def \radius {2.5cm}
  
    \foreach \s in {1,...,15}
    {
      \node (\s) at ({360/(2*\n) + 360/\n * \s}:\radius) {$\s$};
    }
    \node (0) [initial,accepting] at ({360/(2*\n) + 360/\n * \n}:\radius) {$0$};
    \begin{scope}[xshift=6cm]
      \foreach \s in {16,...,31}
      {
        \node (\s) at ({180 + 360/(2*\n) + 360/\n * \s}:\radius)
        {$\s$};
      }
    \end{scope}

    \path (0) edge node {$a$} (1)
    (1) edge node {$b$} (2)
    (2) edge node {$c$} (3)
    (3) edge node {$a$} (4)
    (4) edge node {$c$} (5)
    (5) edge node {$b$} (6)
    (6) edge node {$a$} (7)
    (7) edge node {$b$} (8)
    (8) edge node {$c$} (9)
    (9) edge node {$b$} (10)
    (10) edge node {$a$} (11)
    (11) edge node {$c$} (12)
    (12) edge node {$a$} (13)
    (13) edge node {$b$} (14)
    (14) edge node {$c$} (15)
    (15) edge node {$b$} (0) edge node {$a$} (16)
    (16) edge node {$c$} (17)
    (17) edge node {$b$} (18)
    (18) edge node {$a$} (19)
    (19) edge node {$c$} (20)
    (20) edge node {$a$} (21)
    (21) edge node {$b$} (22)
    (22) edge node {$c$} (23)
    (23) edge node {$b$} (24)
    (24) edge node {$a$} (25)
    (25) edge node {$b$} (26)
    (26) edge node {$c$} (27)
    (27) edge node {$a$} (28)
    (28) edge node {$c$} (29)
    (29) edge node {$b$} (30)
    (30) edge node {$a$} (31)
    (31) edge node {$c$} (0) edge node {$b$} (16);
  \end{tikzpicture}

  \caption{The labeled digraphs $\Lambda_n$ as automata for $n=1,\ldots,5$}
  \label{fig:sf}
\end{figure}
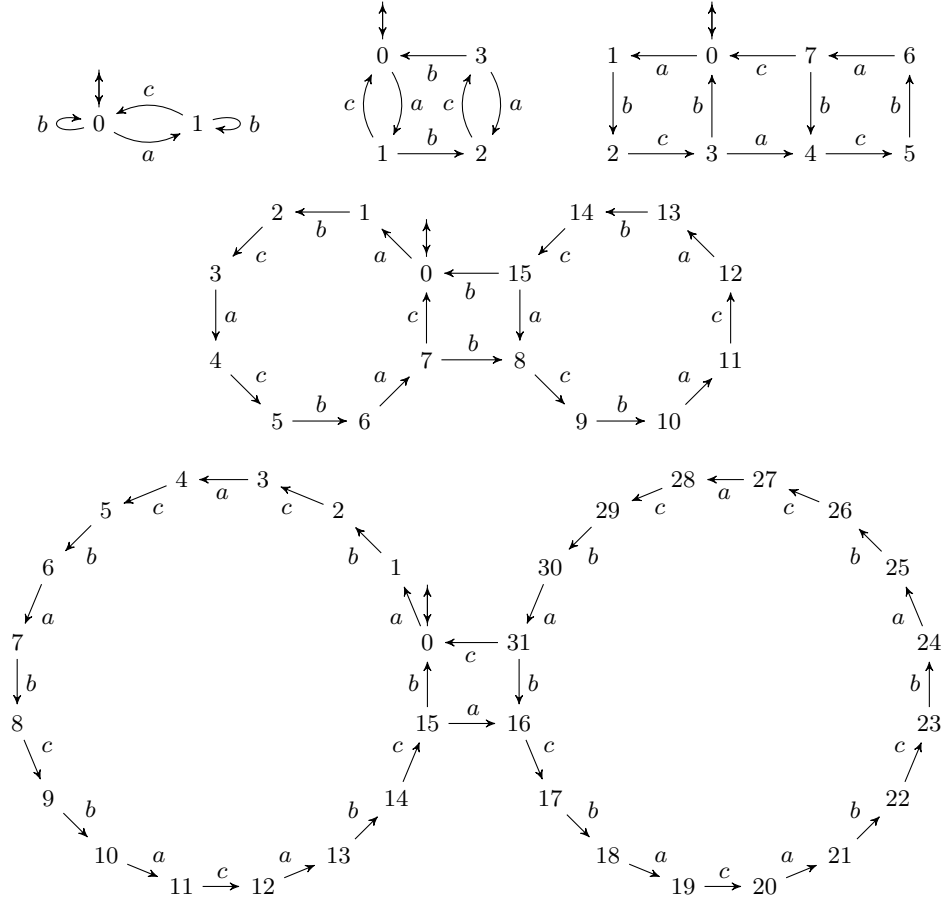

\section*{Acknowledgments}

The first author acknowledges partial funding by CMUP (UID/MAT/
00144/2019) which is funded by FCT (Portugal) with national (MATT'S)
and European structural funds (FEDER) under the partnership agreement
PT2020. %
The work was carried out at Masaryk University, whose hospitality is
gratefully acknowledged, with the support of the FCT sabbatical
scholarship SFRH/BSAB/142872/2018. %
The second author was supported by Grant 19-12790S of the Grant Agency
of the Czech Republic.

\bibliographystyle{amsplain}
\bibliography{sgpabb,ref-sgps}

\end{document}